\documentclass[reqno,11pt]{amsart}
\usepackage{amssymb,latexsym,amsmath}
\usepackage{enumerate}
\usepackage{geometry} 
\usepackage{setspace}
\usepackage{verbatim}
\usepackage{caption}

\usepackage{etoolbox}
\let\bbordermatrix\bordermatrix
\patchcmd{\bbordermatrix}{8.75}{4.75}{}{}
\patchcmd{\bbordermatrix}{\left(}{\left[}{}{}
\patchcmd{\bbordermatrix}{\right)}{\right]}{}{}
\usepackage{hyperref}

\newtheorem{thm}{Theorem}[section]
\newtheorem{cor}[thm]{Corollary}
\newtheorem{lemma}[thm]{Lemma}
\newtheorem{prop}[thm]{Proposition}

\theoremstyle{definition}
\newtheorem{definition}[thm]{Definition}
\newtheorem{example}[thm]{Example}

\theoremstyle{remark}
\newtheorem{remark}[thm]{Remark}

\newtheorem*{note}{Note}

\numberwithin{equation}{section}

\newcommand{\mcal}{\mathcal}
\newcommand{\tb}{\textbf}
\newcommand{\ti}{\textit}
\newcommand{\al}{\alpha}

\newcommand{\ga}{\gamma}
\newcommand{\bbP}{\mathbb{P}}

\newcommand{\dsum}{\displaystyle \sum}
\newcommand{\dprod}{\displaystyle \prod}

\newcommand{\be}{\begin{equation}}
\newcommand{\ee}{\end{equation}}
\newcommand{\ba}{\begin{align*}}
\newcommand{\ea}{\end{align*}}

\newcommand{\x}{\boldsymbol{x}}
\newcommand{\y}{\boldsymbol{y}}
\newcommand{\z}{\boldsymbol{z}}
\newcommand{\bX}{\boldsymbol{X}}
\newcommand{\bY}{\boldsymbol{Y}}
\newcommand{\EY}{E_{\mathcal{Y}_1}}

\newcommand{\bs}{\boldsymbol\sigma}

\newcommand{\bi}{\boldsymbol i}
\newcommand{\lf}{\lfloor}
\newcommand{\rf}{\rfloor}
\newcommand{\lc}{\lceil}
\newcommand{\rc}{\rceil}
\newcommand{\lp}{\left(}
\newcommand{\rp}{\right)}
\newcommand{\lb}{\left[}
\newcommand{\rb}{\right]}

\begin{document}
\title[Random Permutation Graphs]{On the Connected Components of a Random Permutation Graph with a Given Number of Edges}

\author{H\"{u}seyin Acan}
\address{Department of Mathematics, The Ohio State University, 
Columbus, OH 43210}
\email{acan@math.osu.edu}

\author{Boris Pittel}
\address{Department of Mathematics, The Ohio State University, 
Columbus, OH 43210}
\email{bgp@math.osu.edu}

\thanks{Acan's research supported in part by NSF Grant DMS-1101237} 
\thanks{Pittel's research supported in part by NSF Grant DMS-1101237}

\maketitle

\begin{abstract}
A permutation $\bs$ of $[n]$ induces a graph $G_{\bs}$ on $[n]$ -- its edges are inversion pairs in $\bs$. The graph $G_{\bs}$ is connected if and only if $\bs$ is indecomposable. Let $\bs(n,m)$ denote a permutation chosen uniformly at random among all permutations of $[n]$ with $m$ inversions. Let $p(n,m)$ be the common value for the probabilities $\bbP(\bs(n,m) \textrm{ is indecomposable})$ and  $\bbP(G_{\bs(n,m)} \textrm{ is connected})$. We prove that $p(n,m)$ is non-decreasing with $m$ by constructing a Markov process $\{\bs(n,m)\}$ in which $\bs(n,m+1)$ is obtained by increasing one of the components of the inversion sequence of $\bs(n,m)$ by one. We show that, with probability approaching 1, $G_{\bs(n,m)}$ becomes connected for $m$ asymptotic to $m_n=(6/\pi^2)n\log n$. We also find the asymptotic sizes of the largest and smallest components when the number of edges is moderately below the threshold $m_n$.
\end{abstract}

\section{Introduction}\label{Intro}

We denote by $S_n$ the set of permutations of $[n]:=\{1,2,\dots,n\}$. A permutation $\bs=\sigma(1)\sigma(2)\dots\sigma(n) \in S_n$ is \textit{decomposable} if there is a positive integer $k<n$ such that $\{\sigma(1),\sigma(2),\dots,\sigma(k)\}=\{1,2,\dots,k\}$.  If no such $k$ exists, $\bs$ is termed \ti{indecomposable}.
A permutation $\bs$ gives rise to an associated graph $G_{\bs}$
on a vertex set $[n]$, with edge set formed by inversions in $\bs$. That is,
$i<j$ are joined by an edge if and only if  $\bs^{-1}(i)>\bs^{-1}(j)$. The graph $G_{\bs}$ is connected if and only if $\bs$ is indecomposable.  

Indecomposable permutations were first studied by Lentin ~\cite{Lentin} and Comtet ~\cite{Comtet72,Comtet74}. Lentin ~\cite{Lentin} showed that $f(n)$, the number of indecomposable permutations of length $n$, satisfies the recurrence relation 
\[	n!-f(n)=\sum_{i=1}^{n-1}(n-i)!f(i), \quad f(1):=1,			\]
and consequently, $f(n)$ is the coefficient of $t^n$ in the series $1- \big( \sum_{k\ge 0}k!t^k\big)^{-1}$. The same sequence starting with $n=2$ appears in a paper by Hall ~\cite{Hall} in which he shows that the number of subgroups of index $n$ in the free group generated by 2 elements is $f(n+1)$. 

Comtet ~\cite{Comtet72} proved that a permutation $\bs$ chosen uniformly at random among all $n!$ permutations  is indecomposable, whence $G_{\bs}$ is connected, with probability $1-2/n
+O(n^{-2})$. That $\bs$ is indecomposable with high probability (whp), i.e.,
with probability approaching $1$ as $n\to\infty$, should not be too surprising. Indeed,
the number of inversions in a uniformly random $\bs$ is sharply concentrated around its mean value, which is $n(n-1)/4$. So the accompanying  graph $G_{\bs}$ whp has a high edge
density, and as such should be connected whp.
 
Mendez and Rosenstiehl ~\cite{MR} gave a bijection between indecomposable permutations of $[n]$ and pointed hypermaps of size $n-1$. In a recent paper ~\cite{CMR}, Cori et al. proved that the probability $\bbP(\boldsymbol\tau(n,m) \textrm{ is indecomposable})$ is monotone non-increasing in $m$ where $\boldsymbol\tau(n,m)$ denotes a permutation chosen uniformly at random from all permutations of $[n]$ with $m$ cycles. When $m/n \to c$, ($0 < c<1$), they also found the asymptotic probability $p(c)$ of indecomposability  of $\boldsymbol{\tau}(n,m)$.

For more information on  indecomposable permutations, we refer the reader to Comtet ~\cite{Comtet74}, B\'{o}na ~\cite{Bona04}, Flajolet and Sedgewick ~\cite{FS},  Cori ~\cite{Cori}, and the references therein.

In this paper, we study the probability of $\bs(n,m)$ being indecomposable, where $\bs(n,m)$ denotes a permutation chosen uniformly at random from all permutations of $[n]$ with exactly $m$ inversions. In Section 2, we show that the probability $\bbP(\bs(n,m) \textrm{ is indecomposable})$ is non-decreasing in $m$ by finding a Markov process that at $m$-th step produces $\bs(n,m)$ 
from $\bs(n,m-1)$ via increasing one of the components of the {\it inversion sequence\/} of $\bs(n,m-1)$ by one. Counterintuitively, and unlike the Erd\H{o}s-R\'{e}nyi graph process, the set of inversions of $\bs(n,m)$ does not
necessarily contain that of $\bs(n,m-1)$. Despite this, we think that this Markov process may well have other uses in the analysis of $\bs(n,m)$. In Section 3, we find a threshold value $m_n=(6/\pi^2)n\log n$,
for transition
from decomposability to indecomposability of the random permutation $\bs(n,m)$. That is,   $\bs(n,m)$ is decomposable with probability approaching $1$ if $\lim m/m_n<1$, and $\bs(n,m)$
is indecomposable with probability approaching $1$ if $\lim m/m_n>1$. Equivalently,
$m_n$ is the threshold value of the number of edges for connectedness of the 
accompanying permutation graph $G_{\bs(n,m)}$. Notice that $m_n=\Theta(n\log n)$, 
analogously to a classic result of Erd\H{o}s and R\'{e}nyi for $G(n,m)$, the graph chosen
uniformly at random from all graphs with $m$ edges, in which case $m_n=(1/2) n\log n$.
We show further that for 
\[
m=\frac{6n}{\pi^2}\left(\log n+\frac{1}{2}\log\log n+\frac{1}{2}\log (12/\pi)-\frac{12}{\pi^2}+\mu_n\right),  
\quad(|\mu_n|=o(\log\log\log n)),
\]
the number of components of $G_{\bs(n,m)}$ is asymptotically $1+\text{Poisson\,}(e^{-\mu_n})$. 

In Section \ref{se: block sizes}, we demonstrate that, for $\mu_n\to-\infty$ and $|\mu_n|=o(\log\log\log n)$, the
lengths of the largest and the smallest components, scaled by $n$, are asymptotic
to the lengths of the largest and the smallest subintervals in a partition of $[0,1]$ by
$\lfloor e^{-\mu_n}\rfloor$ randomly, and independently,  scattered points. This is quite different than what is observed in the Erd\H{o}s-R\'{e}nyi graph $G(n,m)$, where, whp,  a unique component of order $\Theta(n)$ starts to appear when the number of edges slightly exceeds $n/2$. On the other hand, isolated vertices exist in $G(n,m)$ whp before the graph becomes connected. Hence, in this near-subcritical phase,  
the component sizes in $G_{\bs(n,m)}$ are more balanced compared to component sizes of $G(n,m)$ in a near-subcritical phase.
 
\subsection{Preliminaries}

For a permutation $\bs \in S_n$, a pair $(\sigma(i),\sigma(j))$ is called an \ti{inversion} if $i<j$ and $\sigma(i) > \sigma(j)$. The inversion $(\sigma(i),\sigma(j))$ indicates that the pair $(\sigma(i),\sigma(j))$ is out of order in $\bs$, i.e., the bigger number $\sigma(i)$ appears before the smaller number $\sigma(j)$ in the word notation of the permutation $\bs=\sigma(1)\sigma(2)\dots\sigma(n)$. The \ti{permutation graph} $G_{\bs}$ associated with $\bs$  is the graph with vertex set $[n]$ and edge set corresponding to the inversions of $\bs$; if $(\sigma(i),\sigma(j))$ is an inversion, then $\{\sigma(i),\sigma(j)\}$ is an edge in $G_{\bs}$. A permutation graph can be viewed as an intersection graph induced by a special chord diagram, and that those graphs with no constraint on number of crossings had been studied, notably by Flajolet and Noy \cite{FN}.

Koh and Ree ~\cite{KohRee} showed that $G_{\bs}$ is connected if and only if $\bs$ is  indecomposable. For completeness, here is a proof. One direction is easy.  If a permutation is decomposable, then there is a positive integer $k<n$ such that $\bs([k])=[k]$, so that there is no edge from $[k]$ to $[n]\setminus [k]$ in $G_{\bs}$.  The other direction follows from the observation that if $(a,b)$ is an inversion and $a<c<b$, then either $(a,c)$ is an inversion or $(c,b)$ is an inversion. Equivalently, if $a$ and $b$ are neighbors in $G_{\bs}$ and $a<c<b$, then $c$ is a neighbor of at least one of $a$ and $b$. It follows from this observation that the vertex set of any component of $G_{\bs}$ is a consecutive 
subset of  $[n]$. If $G_{\bs}$ is not connected, then let $k<n$ be the biggest vertex in the component of vertex $1$. Then $\bs([k])=[k]$, which means that $\bs$ is  decomposable. 

More generally,  let $\bs=\sigma(a)\dots \sigma(b)$ be a permutation of the set $\{a,a+1,\dots,b\}$;  $\bs$ is called decomposable if 
\[	 \{\sigma(a),\sigma(a+1),\dots,\sigma(k)\}=\{a,a+1,\dots,k\}	\]
for some $a\le k \le b-1$, and it is called indecomposable otherwise. Any permutation $\bs$ of $[n]$ can be decomposed into its \textit{indecomposable blocks}, where each block is an indecomposable permutation of some consecutive set of numbers. The indecomposable blocks of $\bs$ correspond to the connected components of $G_{\bs}$. We write $\bs = (\bs^1, \bs^2, \dots, \bs^l)$ where $\bs^i$ is an indecomposable permutation of $\{k_{i-1}+1,k_{i-1}+2,\dots,k_i\}$ for some integers $0=k_0<k_1<k_2<\cdots <k_l=n$. For example, if $\bs= 24135867$, then $\bs=(\bs^1,\bs^2,\bs^3)$, where $\bs^1=2413$, $\bs^2=5$, and $\bs^3=867$.

We denote the set of the permutations of $[n]$ with $m$ inversions by $\mathcal{S}(n,m)$, and the cardinality of $\mcal S(n,m)$ by $s(n,m)$. 
The generating function for the numbers $s(n,m)$ was shown by Muir ~\cite{Muir} to satisfy
\begin{equation}\label{Snx=}	\mathcal{S}_n(x):= 	\sum_{m \ge 0}s(n,m)x^m= \prod_{i=0}^{n-1}(1+x+\cdots+x^{i}).
\end{equation}			
Probabilistically, this product-type formula means that the number of inversions of the
uniformly random permutation of $[n]$ equals, in distribution, to $\sum_{i=0}^{n-1}X_i$,
where $X_i$ is uniform on $\{0,1,\dots,i\}$, and $X_0,\dots,X_{n-1}$ are independent. 
Using the formula above for $\mathcal{S}_n(x)$, many asymptotic results were found for $s(n,m)$, see for instance Bender ~\cite{Bender},  Clark ~\cite{Clark}, Louchard and Prodinger  ~\cite{LP}, and Margolius ~\cite{Margolius}. From formula~\eqref{Snx=} we obtain the recurrence
\begin{equation} \label{eq: recurrence s(n,m)}
s(n,m)= s(n-1,m)+\cdots+s(n-1, m-(n-1)),
\end{equation}
where $s(n-1,i)=0$ if $i<0$.

We consider $\mcal S(n,m)$ as a probability space equipped with the uniform distribution, denoting the random element of this space by $\bs(n,m)$. If $m< n-1$, then $\bs(n,m)$ is necessarily decomposable since any graph with $n$ vertices and $m$ edges is disconnected if $m<n-1$. Similarly, if $m > {n-1 \choose 2}$, then $\bs(n,m)$ is necessarily indecomposable. Therefore we only need to consider the values of $m$ between $(n-1)$ and ${n-1 \choose 2}$. 

A key element in our proofs is a classic notion of the \textit{inversion sequence} of a permutation. For a permutation $\bs = \sigma(1)\sigma(2)\dots \sigma(n)$, the inversion sequence of $\bs$ is 
$\x=\x(\bs) =x_1x_2\ldots x_n$, where 
\[	x_i= | \{j: j<i \textrm{ and }\sigma(j)>\sigma(i) \} |	.	\]
In words, $x_i$ is the number of inversions involving $\sigma(i)$ and the elements of $\bs$ preceding $\sigma(i)$. It is clear from the definition that 
\be \label{s1:e1}
0 \leq x_i \leq i-1 \quad (1\leq i\leq n).	
\ee
There are exactly $n!$ integer sequences of length $n$ meeting the constraint ~\eqref{s1:e1}. In fact, every sequence $\boldsymbol x$ satisfying ~\eqref{s1:e1} is an inversion sequence of a permutation, so that there is a bijection between the set of permutations and the set of sequences $\x$ satisfying ~\eqref{s1:e1}. Hence we have a bijective proof of \eqref{Snx=}.

Here is how a permutation $\bs$ is uniquely recovered from its inversion sequence $\x$. First of all, $\sigma(n)=n-x_n$. Recursively, if $\sigma(n),\sigma(n-1),\dots,\sigma(t+1)$ have been determined, then $\sigma(t)$ is the $(1+x_t)$-th largest element in the set $[n]\setminus \{\sigma(n), \sigma(n-1),\dots, \sigma(t+1) \}$. For more information on inversion sequences see Knuth \cite[Section 5.1.1]{Knuth}.

\begin{example}\label{IS}
Let $\x= 002012014$. The permutation $\bs$ corresponding to this sequence is a permutation of $[9]$. Then $\sigma(9)=9-4=5$. To find $\sigma(8)$ we need to find the second largest element of $\{1,2,3,4,6,7,8,9\}$, which is $8$. To find $\sigma(7)$, we need to find the  largest element of the set $\{1,2,3,4,6,7,9\}$, which is $9$. If we continue in the same manner, we get $\bs = 231764985$.
\end{example}
Let $\|\boldsymbol x\|:=x_1+x_2+\dots+x_n$.
The number of inversions in $\bs$ is equal to $||\x||$, so  the set $\mcal S(n,m)$ is, bijectively, the set of $\x$'s meeting ~\eqref{s1:e1} and 
\be \label{s1:e2}
	||\x||=m.		
\ee

We denote the set of sequences $\x$ satisfying ~\eqref{s1:e1}--\eqref{s1:e2} by $\mcal X(n,m)$. The bijection enables us to identify $\bs(n,m)$, the uniformly random permutation of $[n]$, with $\bX=\bX(n,m)$, chosen uniformly at random from $\mcal X(n,m)$.

To conclude the preliminaries, we note that a permutation $\bs=\sigma(1)\sigma(2)\dots\sigma(n)$ is decomposable if and only if there exists $k<n$ such that  its inversion sequence $\x$ has a tail $x_{k+1}x_{k+2} \dots x_n$ which is an inversion sequence of a permutation of $[n-k]$, see Cori et al. ~\cite{CMR}.

\begin{remark}
The permutation $\bs(n,m)$ has a different distribution than the permutation obtained after $m$-th step in the random sorting network, where, at each step, two adjacent numbers in correct order are chosen uniformly randomly and they are swapped. Although the choice at each step is uniform in this random sorting network, the resulting permutation after $m$ steps is not uniform among all permutations with exactly $m$ inversions.
\end{remark}

\section{A Markov Process}\label{Process}

So far, each uniformly random $\bs(n,m)$ has been defined separately, on its own
probability space $\mcal S(n,m)$. Our goal in this section is to build a Markov process that produces $\bs(n,m)$ from $\bs(n,m-1)$ for each $m$. In view of bijection between $\mcal S(n,m)$ and $\mcal X(n,m)$, it suffices to construct a Markov process $\{\bX(\mu)\}_{\mu\ge 0}=\{\bX(n,\mu)\}_{\mu\ge 0}$  in such a way that each  $\bX(\mu)$ is distributed uniformly on $\mcal X(n,\mu)$, the set of solutions of ~\eqref{s1:e1}--\eqref{s1:e2} with $m=\mu$, and $\bX(\mu+1)$ is obtained by adding $1$ to one of the components $\bX(\mu)$ according to a (conditional) probability distribution $\bold p(\bX(\mu))$. 

It is convenient to view  such a process as a dynamic allocation scheme.
Formally, there are $n$ boxes numbered $1$ through $n$,  and $\binom{n}{2}$ indistinguishable balls.
Box $i$ has capacity $i-1$, i.e., it can accept at most $i-1$ balls. Recursively, after $t-1$ steps
the occupancy numbers are ${\bX}(t-1)=X_1(t-1)\dots X_n(t-1)$, satisfying ~\eqref{s1:e1}--\eqref{s1:e2}
for $m=t-1$, and we throw the $t$-th ball  into one of the boxes according to a probability distribution 
$\bold p(\bX(t-1))=\{p_i(\bX(t-1))\}_{i\in [n]}$. Obviously, $p_i(\x)=0$
if $x_i=i-1$. Once $\bold p(\cdot)$ is defined, we obtain a Markov process $\{\bX(t)\}_{t\ge 0}$.
We have to show  existence of an admissible $\bold p(\x)$ such that, if
$\bX(0)=\boldsymbol 0$, then 
for every $m$, $\bX(m)$ is distributed uniformly on $\mcal X(n,m)$.

The proof is by induction on $n$. We begin with a reformulation of the problem in terms of an one-step transition matrix.  

\subsection{Basic definitions and observations}
An equivalent formulation of the problem is as follows.
For every $n$ and $0\le m<\binom{n}{2}$, we have to find a probability transition
matrix $\rho=\rho_{n,m}$. The matrix $\rho$ has $s(n,m)$ rows and $s(n,m+1)$ columns 
indexed by the elements of $\mcal X(n,m)$ and the elements of $\mcal X(n,m+1)$, respectively. For $\x \in \mcal X(n,m)$ and $\y \in \mcal X(n,m+1)$, we say that $\y$ \textit{covers} $\x$ if there is an index $j$ such that $y_j=x_j+1$ and $y_i=x_i$ for $j\neq i$.

The entries $\rho(\x,\y)$, where $\x \in \mcal X(n,m)$ and $\y \in \mcal X(n,m+1)$, have to meet three conditions, a
trivial one
\be \label{s2:e1}
\sum_{\y}\rho(\x,\y)=1, 
\ee
the uniformity preservation condition
\be\label{s2:e2}
\frac{1}{s(n,m)}\sum_{\x}\rho(\x,\y)=\frac{1}{s(n,m+1)},
\ee
and if $\y$ does not cover $\x$, then
\be \label{eq: third condition}
\rho(\x,\y)=0.
\ee

For illustration consider the two simplest cases.

\tb{Case $\boldsymbol{n=2}$}.  Necessarily $m=0$, and $\rho_{2,0}$ is a $1\times 1$ matrix with entry equal
$1$.

\tb{Case $\boldsymbol{n=3}$}.  Here $m$ can take the values $0,1$, and $2$. We have $s(3,0)=1$, $s(3,1)=2$, $s(3,2)=2$,  and $s(3,3)=1$. We present the matrices $\rho_{3,0}$, $\rho_{3,1}$, and $\rho_{3,2}$ below.
\[
\rho_{3,0} =\quad \bbordermatrix{\text{}&010&001\cr
                000 & 1/2 &1/2}, \quad 
\rho_{3,1} =\quad \bbordermatrix{\text{}&011&002\cr
                010& 1 & 0 \cr
                001& 0 & 1}, \quad
\rho_{3,2} =\quad  \bbordermatrix{\text{} &012\cr
                011 & 1 \cr
		  002 & 1}				
\]

\begin{lemma}\label{Symmetry}
$s(n,m) = s(n,{n \choose 2}-m)$.
\end{lemma}

\begin{proof}
There is a bijection between $\x\in \mcal X(n,m)$ and 
$\y \in \mcal X \left(n,\binom{n}{2}- m\right)$ given by
\[
x_1\dots x_n\leftrightarrow (0-x_1)\dots (n-1-x_n):=y_1\dots y_n.
\qedhere \]
\end{proof}

\begin{lemma}\label{halfm}
If $\rho_{n,m}$ exists, then so does $\rho_{n,\tilde m}$ for $\tilde m={n \choose 2}-1-m$.
\end{lemma}

\begin{proof}
By equations  ~\eqref{s2:e1}--\eqref{s2:e2}, the row sums and the column sums of $\rho(n,m)$ 
are $1$ and $s(n,m)/s(n,m+1)$, respectively. Given $\z$, meeting  ~\eqref{s1:e1}, define $\z^\prime=
(0-z_1)\dots(n-1-z_n)$. Then, for $\x\in \mcal X(n,\tilde{m})$, $\y\in\mcal X(n,\tilde{m}+1)$,
we have $\y^\prime\in \mcal X(n,m)$, $\x^\prime\in \mcal X(n,m+1)$, and 
$\y$ covers $\x$ if and only if $\x^\prime$ covers $\y^\prime$.  So
we set
\be \label{s2:e3}
\rho_{n,\tilde m}(\x,\y):=\rho_{n,m}(\y^\prime,\x^\prime)\cdot\frac{s(n,m+1)}{s(n,m)}.
\ee
In matrix form, we have $\rho_{n,\tilde m}= \rho_{n,m}^T\cdot \frac{s(n,m+1)}{s(n,m)}$, where $\rho_{n,m}^T$ is the transpose of $\rho_{n,m}$. By  ~\eqref{s2:e1}--\eqref{s2:e3}, the row sums of $\rho_{n,\tilde m}$ are $1$, and the column sums are
$s(n,m+1)/s(n,m)=s(n,\tilde m)/s(n,\tilde m+1)$, see Lemma ~\ref{Symmetry}.
\end{proof}

\begin{thm}\label{rhoexists}
The stochastic matrices $\rho_{n,m}$ exist for all integers $n\geq 2$ and $0\leq m \leq {n \choose 2}-1$.
\end{thm}

\begin{note}
Before we start our proof we make a comment why we should expect to find the matrices $\rho_{n,m}$. The entries of $\rho_{n,m}$ must satisfy [$s(n,m)+s(n,m+1)$] equations given in  \eqref{s2:e1} and \eqref{s2:e2}. On the other hand, by equation \eqref{eq: third condition}, the number of variables $\rho(\x,\y)$ is equal to $\sum_{\ell=1}^{n-1}\ell\cdot s(n,m+1,\ell)$, where $s(n,m+1,\ell)$ is the number of elements in $\mcal X(n,m+1)$ with exactly $\ell$ nonzero terms. For $m\ge 1$ and $n\ge 4$, it is not hard to see that the number of variables is bigger than the number of equations that must be satisfied. In fact, for large values of $m$ and $n$, the number of variables is much bigger than the number of equations.
\end{note}

\begin{proof}[Proof of Theorem \ref{rhoexists}]
In our proof, which is by induction on $n$, we reduce the number of variables and consider more special types of matrices.
The basis of induction is  the cases $n=2$, $n=3$ already considered 
above. Inductively, let $n\ge 4$ and suppose that the matrices $\rho_{n-1,m^\prime}$ exist for all possible values of $m^\prime$. By Lemma ~\ref{halfm}, we only need to prove that $\rho_{n,m}$ exists for an arbitrary $m<{n \choose 2}/2$, so  in the rest of the proof we assume $m<{n\choose 2}/2$.

We denote by $\preceq$ the reverse lexicographic order on $n$-long integer sequences. Thus, if $\x=x_1x_2\dots x_n$ and $\y=y_1y_2\dots y_n$ are two sequences, then $\boldsymbol x\preceq \boldsymbol y$  if and only if   $y_i>x_i$ for $i=\max\{j\,:\,y_j\neq x_j\}$.  For example $0110\textbf{2}32 \preceq 0010\textbf{3}32$, and $i=5$. 

We now introduce a matrix $M$ that will serve as $\rho_{n,m}$ for an appropriate choice of parameters. Its $s(n,m)$ rows  and $s(n,m+1)$ columns  are labeled by the inversion sequences $\boldsymbol x$, with $\|\boldsymbol x\|=m$, and  by the inversion
sequences $\boldsymbol y$,
with $\|\boldsymbol y\|=m+1$ resp.,  both rows and columns being listed in the increasing order with respect to $\preceq$. 
Specifically,
\begin{enumerate}[(i)]
\item $M(\boldsymbol x,\boldsymbol y)=0$ if $\boldsymbol y$ does not cover $\boldsymbol x$;
\item if $\boldsymbol y$ covers $\boldsymbol x$, and $y_n=i=x_n+1$, then $M(\boldsymbol x,\boldsymbol y)=\beta_i$, where $\beta_i$ is to be determined for $1\leq i\leq n-1$;
\item if $\boldsymbol y$ covers $\boldsymbol x$, and $y_n=x_n=j$, then 
\[	 M(\boldsymbol x,\boldsymbol y)=(1-\beta_{j+1})\cdot \rho_{n-1,m-j}(\boldsymbol{\hat x},\boldsymbol{\hat y}),\quad \boldsymbol{\hat x}:=x_1\dots x_{n-1},\,
\,\boldsymbol{\hat y}:=y_1\dots y_{n-1}, \]
\end{enumerate}
with $\beta_n:=0$. $\beta_j$ has the meaning of (conditional) probability that the next inversion sequence is obtained by increasing the last component of the current one. With complementary probability $1-\beta_j$ we change one of the other components of the current inversion  sequence using the transition matrices for $n-1$. We need to determine the probabilities $\beta_1,\dots,\beta_r$ from \textit{exactly} $r$ equations, where $r:=\min\{n-1,m+1\}$. 

Note that, condition (iii) guarantees that all the row sums are equal to 1. Since $M(\boldsymbol x,\boldsymbol y)=0$ unless $\boldsymbol y$ covers $\boldsymbol x$, $M$ has a two-diagonal block structure, see Figure ~\ref{figure1}. (Some of the rows and columns are empty if $m<n-1$.) From (iii),  the 
diagonal of $M$, that starts at the left uppermost block,  consists of the submatrices  $\rho^{\prime}_{n-1,m-j}:=(1-\beta_{j+1})\rho_{n-1,m-j}$, of order $s(n-1,m-j)\times s(n-1,m+1-j)$. The second, upper, diagonal of $M$
consists of the matrices $\beta_j I_j$, where $I_j$ is an identity matrix, of order
$s(n-1,m+1-j)\times s(n-1,m+1-j)$.   All the other block-submatrices are $0$ matrices. 
Clearly if  there exist  $\beta_1,\beta_2,\dots,\beta_{n-1}\in [0,1]$ such that $M$ is a stochastic matrix,  then $M$ is a sought-after $\rho_{n,m}$.

\begin{figure}\caption*{}
\begin{center}
\begin{tabular}{r|c|c|c|c|c|c|}
	\multicolumn{1}{c}{}
&	\multicolumn{1}{c}{$y_n=0$}
&	\multicolumn{1}{c}{$y_n=1$} 
&	\multicolumn{1}{c}{$y_n=2$} 
&	\multicolumn{1}{c}{$\qquad \dots \qquad$}
&	\multicolumn{1}{c}{$y_n=n-2$} 
&	\multicolumn{1}{c}{$y_n=n-1$}   \\
\cline{2-7}
$x_n=0$ & $\rho^{\prime}_{n-1,m}$ & $\beta_1 I$ 	&	  &  & &			\\
\cline{2-7}
$x_n=1$ & & $\rho^{\prime}_{n-1,m-1}$  & $\beta_2 I$    &	 &	 & 			\\
\cline{2-7}
$x_n=2$ & &   &  $\rho^{\prime}_{n-1,m-2}$ & $\ddots$	 &	 &		\\
\cline{2-7}
$\vdots$ &  &   & & $\ddots$ & $\ddots$   &									\\
\cline{2-7}
$x_n=n-2$ &  & & 	&	 &$\rho^{\prime}_{n-1,m-n+2}$ & $\beta_{n-1} I$				\\
\cline{2-7}
$x_n=n-1$ &  &   & 	&	 &	 &$\rho^{\prime}_{n-1,m-n+1}$ 				\\
\cline{2-7}
\end{tabular}
\caption{Block structure of $M$ for $m\ge n-1$.} \label{figure1}
\end{center}
\end{figure}

Before we go into details, let us illustrate how $\rho_{4,2}$ is obtained from $\rho_{3,0}$, $\rho_{3,1}$, and $\rho_{3,2}$. The matrix $M$ will have $5$ rows and $6$ columns since $s(4,2)=5$ and $s(4,3)=6$. The sequences corresponding to the rows are $0110,0020,0101,0011,0002$, and the sequences corresponding to the columns are $0120,0111,0021,0102,0012,0003$. The submatrices $\rho^{\prime}_{3,2}$, $\rho^{\prime}_{3,1}$, and $\rho^{\prime}_{3,0}$ are $2\times 1$, $2\times 2$, and $1\times 2$ matrices, respectively. Likewise, the matrices $\beta_1 I$, $\beta_2 I$, and $\beta_3 I$ are $2\times 2$, $2\times 2$, and $1\times 1$ matrices, respectively. Using the matrices for $n=3$ we see that $M$ has the following structure.

\[
M =\quad \bbordermatrix{  & 0120 & 0111 & 0021 & 0102 & 0012 & 0003\cr 
                0110 & 1-\beta_1& \beta_1 & 0 & 0 & 0 & 0 \cr
                0020 & 1-\beta_1 & 0 & \beta_1 & 0 & 0 & 0  \cr
		 0101 & 0 & 1-\beta_2 & 0 & \beta_2 & 0 & 0  \cr
		 0011 & 0 & 0 & 1-\beta_2 & 0 & \beta_2 & 0  \cr	
		 0002 & 0 & 0 & 0 & (1-\beta_3)/2 & (1-\beta_3)/2 & \beta_3		}
\]\\

For $M$ to serve as $\rho_{4,2}$, $M$'s row sums must be $1$, and $M$'s column sums must be $5/6$. Solving the resulting linear equations, we obtain  $\beta_1=7/12$, $\beta_2=9/12$, and $\beta_3=10/12$. Hence we have

\[
\rho_{4,2} =\quad \bbordermatrix{  & 0120 & 0111 & 0021 & 0102 & 0012 & 0003\cr 
                0110 & 5/12 & 7/12 & 0 & 0 & 0 & 0 \cr
                0020 & 5/12 & 0 & 7/12 & 0 & 0 & 0  \cr
		0101 & 0 & 3/12 & 0 & 9/12 & 0 & 0  \cr
		0011 & 0 & 0 & 3/12 & 0 & 9/12 & 0  \cr	
		0002 & 0 & 0 & 0 & 1/12 & 1/12 & 10/12		}
\]\\

Turn to the general case. There are two possibilities: $m \ge n-1$ or $m \le n-2$. The system of equations are slightly different in the two cases.

\textbf{Case 1:} $m \geq n-1$.

Recall that $\beta_n=0$, and let $\beta_0:=0$ also. The row sums in $M$ are $1$ by (iii). As
for the column sums, we have 
\[	\sum_{\x}M(\x,\boldsymbol y)= \beta_i+ (1-\beta_{i+1})\ga_i, \quad \textrm{ if } y_n=i, \quad 0\leq i \leq n-1,		\]
where
\[ 		 \ga_i:=\dfrac {s(n-1,m-i)}{s(n-1,m-i+1)}, 	\quad  0\leq i \leq n-1,	\]
represents the column sum of the stochastic matrix $\rho(n-1,m-i)$.
If we set 
\[	\ga:=\dfrac {s(n,m)}{s(n,m+1)}, 	\] 
we need to show that the system of equations
\begin{equation}\label{eq: betas}
\beta_{k-1}+(1-\beta_{k})\gamma_{k-1}= \gamma, \quad 1\leq k \leq n 
\end{equation}
has a solution $\beta_1,\beta_2,\dots,\beta_{n-1}\in [0,1]$. In this system of equations there are $n-1$ unknowns and $n$ equations. Obviously the first $n-1$ equations have a unique solution $\beta_1, \beta_2,\dots, \beta_{n-1}$. To show
that this solution satisfies the last equation, we need to prove that this equation
is just a linear combination of the first $n-1$ equations. 

To this end, for $1\leq k \leq n-1$, we multiply the $k$-th equation in \eqref{eq: betas} by $s(n-1, m+2-k)$  to obtain
\[		\lb \beta_{k-1}	+ (1-\beta_k)\ga_{k-1} \rb \cdot s(n-1,m+2-k) = \ga \cdot s(n-1,m+2-k).		\tag{$k$}	\]
Then we add the equations $(1),\dots, (n-1)$. The right hand side (RHS) of the resulting
equation is
\[ 		\textrm{RHS}= \ga \cdot \dsum_{k=1}^{n-1}s(n-1,m+2-k) 	= \ga \cdot \dsum_{k=0}^{n-2}s(n-1,m+1-k).		\]
Using the recurrence relation  (see equation~\eqref{eq: recurrence s(n,m)})
\[		s(p,q)= \dsum_{i=0}^{p-1}s(p-1,q-i),		\]
we simplify the expression above to  
\[	\textrm{RHS}= \ga \cdot \lb s(n,m+1)-s(n-1,m+2-n)\rb = s(n,m) -\ga \cdot s(n-1,m+2-n).			\]
As for the left hand side of the resulting equation (LHS),  we have
\[	\textrm{LHS}= \dsum_{k=1}^{n-1} \beta_{k-1} \cdot s(n-1, m+2-k) +\dsum_{k=1}^{n-1}(1-\beta_k)\ga_{k-1}\cdot s(n-1, m+2-k).			\]
Plugging $\ga_i= \frac{s(n-1,m-i)}{s(n-1,m+1-i)}$ into the second sum, and using $\beta_0=0$
together with the recurrence relation for $s(p,q)$, 
we can write the above equation as
{\allowdisplaybreaks
\begin{align*}
\textrm{LHS}	&= \dsum_{k=1}^{n-1} \beta_{k-1}  s(n-1, m+2-k) +\dsum_{k=1}^{n-1}(1-\beta_k)s(n-1, m+1-k)\\
	&=  (1-\beta_{n-1})s(n-1,m+2-n) + \dsum_{k=1}^{n-2} s(n-1,m+1-k) \\
	&= (1-\beta_{n-1}) s(n-1,m+2-n) +  s(n,m)-s(n-1,m+2-n)-s(n-1,m+1-n) \\
	&= s(n,m) -s(n-1,m+1-n)- \beta_{n-1}s(n-1,m+2-n).
\end{align*}}
By equating LHS and RHS, the resulting equation becomes
\[	\ga\cdot s(n-1,m+2-n)= s(n-1,m+1-n)+ \beta_{n-1}\cdot s(n-1,m+2-n).		\]
Dividing this equation by $s(n-1,m+2-n)$ we get 
\[	\gamma= \beta_{n-1}+\gamma_{n-1},									\] 
which corresponds to the last equation in \eqref{eq: betas}.

We have proved that the system of equations \eqref{eq: betas} has a unique solution. 
It remains to show that  $\beta_1,\beta_2,\dots,\beta_{n-1}$ are all in $[0,1]$.

The sequence $\{s(n,m)\}$ is known to be log-concave, i.e., $s(n,m-1)s(n,m+1)\leq s(n,m)^2$ for $m=1,2,\dots,{n \choose 2}-1$. (For a combinatorial proof see B\'{o}na ~\cite{Bona05}). Thus, we have 
\[		\ga_i= \frac{s(n-1,m-i)}{s(n-1,m+1-i)} \geq 	\frac{s(n-1,m-i-1)}{s(n-1,m-i)} =\ga_{i+1},	\]
so that  
\be \label{s2:e4}	 \ga_0\geq \ga_1\geq \dots \geq \ga_{n-1} \geq 0.		\ee
The log-concavity of $\{s(n,m)\}$ and Lemma ~\ref{Symmetry} together imply that $\{s(n,m)\}$ is unimodal with the maximum term corresponding to $m=\lf {n \choose 2}/2 \rf$ or $m=\lc {n \choose 2}/2 \rc$. Since  $m<\binom{n}{2}/2$, we have
\be\label{s2:e5}	
0\le \gamma= \frac{s(n,m)}{s(n,m+1)} \le 1.		
\ee

Suppose first that $\beta_i >1$ for some $1\leq i\leq n-1$. Solving for $\beta_{i+1}$ in the equation 
$	\beta_{i}+(1-\beta_{i+1})\gamma_i=\gamma	$
we get
\[	\beta_{i+1} = 1+ \frac{\beta_i-\ga}{\ga_i}	>1.	\] 
Iterating this $n-1-i$ times, we get 
$\beta_{n-1} > 1$, which contradicts the equation $\beta_{n-1}+\ga_{n-1}=\ga$. Hence we must have $\beta_i \leq 1$ for all  $i\in \{1,2,\dots,n-1\}$. What's left is to show that $\beta_i$'s are all
non-negative.

By equations ~\eqref{s2:e4}--\eqref{s2:e5} there exists $i^* \in \{0,1,\dots,n\}$ such that $\ga_{i^*-1} \geq \ga \geq \ga_{i^*}$ (we can take $\ga_{-1}=1$ and $\ga_n=0$ if necessary). We have three cases to consider.

\begin{enumerate}[(1)]
\item  \label{firstcase}
$\boldsymbol{i^\ast=0}$.\\
We solve for $\beta_i$'s moving backward, starting with $\beta_{n-1}$.  By the last equation in \eqref{eq: betas}, we have $\beta_{n-1}=\ga-\ga_{n-1}\geq 0$. Using \eqref{eq: betas}, if $\beta_k \geq 0$, then
\[	\beta_{k-1} = \ga -(1-\beta_k)\ga_{k-1} \geq \ga -(1-\beta_k)\ga  = \beta_k\gamma  \geq 0.	\] 
Inductively, we get $\beta_{n-2},\beta_{n-3},\dots,\beta_1 \geq 0$.
\item $\boldsymbol{i^\ast=n}$.\\
 We solve for $\beta_i$'s moving forward, starting with $\beta_1$.  By the first equation in \eqref{eq: betas}, we have $\beta_1 = 1-\frac{\ga}{\ga_0} \geq 0$. If $\beta_k \geq 0$, then $\beta_{k+1}= \frac{\beta_k+ \ga_k-\ga}{\ga_k} \geq 0$ by \eqref{eq: betas}. Again, we get $\beta_1,\beta_2,\dots,\beta_{n-1} \geq 0$ inductively.
 \item $\boldsymbol{0< i^\ast< n}$.\\
In this case, starting with $k=1$ and moving forward, we first use the equation $\beta_{k}= \frac{\beta_{k-1}+ \ga_{k-1}-\ga}{\ga_{k-1}}$  to show that $\beta_k\geq 0$ for $k=1,2,\dots,i$. Then, we start from the last equation in \eqref{eq: betas} and go backwards (as in the case $i^\ast=0$) to see that $\beta_{i+1},\beta_{i+2},\dots,\beta_{n-1}$ are all non-negative.
\end{enumerate}

\textbf{Case 2:} $m \leq n-2$.

In this case $M$ has a block structure, with the submatrices $\rho^{\prime}_{n-1,m},\rho^{\prime}_{n-1,m-1},\dots,\rho^{\prime}_{n-1,0}$, and $ \beta_1I,\beta_2I,\dots,\beta_{m+1}I$,  forming
two block diagonals, with all other blocks filled with zeroes, see Figure ~\ref{figure1}  (the blocks corresponding to $x_n>m$ or $y_n>m+1$ do not exist).

The equations for the parameters $\beta_0=0,\beta_1,\dots,\beta_{m+1}$ are as follows:
\begin{equation}\label{eq: betas2}
\beta_{k-1}+(1-\beta_{k})\gamma_{k-1}= \gamma, \quad 1\leq k \leq m+1, \qquad \text{and} \qquad \beta_{m+1}=\gamma 
\end{equation}

As in the previous case, we multiply the equation $\beta_{k-1}+(1-\beta_{k})\gamma_{k-1}= \gamma$ in \eqref{eq: betas2} by $s(n-1,m+2-k)$ for $1\leq k \leq m+1$, add the resulting $m+1$ equations, and use equation~\eqref{eq: recurrence s(n,m)},
\[s(n,m)= s(n-1,m)+s(n-1,m-1)+\cdots+ s(n-1,0), \]
to simplify the resulting LHS and RHS. Equating the LHS and RHS, we arrive at the last equation in \eqref{eq: betas2}, which is $\beta_{m+1}=\gamma$. So, the system of equations \eqref{eq: betas2} has a unique solution. Then, arguing as in the case
$m\ge n-1$,  we show that all $\beta_i$'s lie in $[0,1]$. 

In either case, we set  $\rho_{n,m}=M$ for the corresponding parameters $\beta_i$'s.
The proof of Theorem \ref{rhoexists} is complete.
\end{proof}

\begin{definition}\label{decompositionpoint}
Let $\boldsymbol a=a_1a_2\dots a_n$ be a sequence of non-negative integers. For $j \in [n-1]$, we will say that $\boldsymbol a$ is \ti{decomposable at $j$} if $a_{j+i}\leq i-1$ for every $i \in [n-j]$. In this case, we also say that $j$ is a \textit{decomposition point}.
\end{definition}

If $\boldsymbol a$ is decomposable at $j$, then the tail $a_{j+1}a_{j+2}\dots a_n$ is an inversion sequence of a permutation. Note that a permutation $\bs$ is decomposable at $j$ if and only if its inversion sequence $\boldsymbol x(\bs)$ is decomposable at $j$. The sequence $\x= 002012014$ is decomposable at $3$ since $012014$ is the inversion sequence of the permutation $431652$.

The following theorem is a direct corollary  of Theorem \ref{rhoexists}.

\begin{thm}\label{Monotonicity}
\mbox{}
\begin{enumerate}[(i)]
\item $\bbP (\bs (n,m) \textrm{ is indecomposable}) \leq \bbP (\bs(n,m+1) \textrm{ is indecomposable})$.\\
\item More generally, $C(\bs(n,m+1))$ is stochastically dominated by $C(\bs(n,m))$, where $C(\bs)$ denotes the number of blocks in $\bs$, that is,
\[
\bbP[C(\bs(n,m+1))\geq j]\le \bbP[C(\bs(n,m))\geq j],\quad\forall\,j\in [n].
\]
\end{enumerate}
\end{thm}

\begin{proof}
We couple the two random permutations $\bs(n,m)$ and $\bs(n,m+1)$ by considering their inversion sequences as two, consecutive, snapshots of the process described above. A  permutation $\boldsymbol \tau$ of $[n]$ is decomposable if and only if there exists some $j\in [n-1]$ such that $\boldsymbol \tau$'s inversion sequence is decomposable at $j$. Note that the inversion sequence of $\bs(n,m+1)$ covers the inversion sequence of $\bs(n,m)$, i.e., they agree at every position except a single $k\in [n]$, where $x_k(\bs(n,m+1))=x_k(\bs(n,m))+1$.  Consequently, if $\bs(n,m+1)$ is decomposable at some position $j$, then so is $\bs(n,m)$. This finishes the proof of $(i)$. The proof of
$(ii)$ is similar.
\end{proof}

\section{Phase Transition}

Our goal in this section is to identify a threshold value $m_n$ for $\bs(n,m)$
to become indecomposable with high probability (whp), i.e., with probability approaching 
$1$ as $n$ tends to infinity. In the rest of the paper, $\alpha$ denotes the fraction $m/n$.

\begin{thm}\label{main} Let  $C(\bs(n,m))$ denote the number of components of $G_{\bs(n,m)}$.
Let 
\begin{equation}\label{m/n=}
\al= \frac{6}{\pi^2}\left( \log n +\frac{1}{2}\log\log n+ \frac{1}{2}\log (12/\pi)-\frac{\pi^2}{12}+\mu_n\right).
\end{equation}
If $|\mu_n|=o(\log\log\log n)$ as $n \to \infty$,  then
$C(\bs(n,m))-1$ is asymptotic in distribution, to a Poisson random variable with mean $e^{-\mu_n}$.
More precisely, denoting by $d_{TV}(X,Y)$ the total variation distance between the distributions
of $X$ and $Y$,
\[
d_{\textup{TV}}\bigl[C(\bs(n,m))-1,\textup{Poisson}(\lambda_n)\bigr]\le (\log n)^{-1+\varepsilon},\quad\forall\,
\varepsilon>0,
\]
where
\[	\lambda_n= n\prod_{j=1}^{\infty} \left( 1-\frac{\alpha^j}{(\alpha+1)^j}\right) =\exp\left[ -\mu_n+O\left(\frac{\log\log n}{\log n}\right)\right].		\]
\end{thm}

\tb{Note:} Combining Theorem ~\ref{main} and Theorem ~\ref{Monotonicity}, we conclude:
(i) Let $\mu_n\to \infty$ in \eqref{m/n=};  then whp $C(\bs(n,m))=1$, i.e., $\bs(n,m)$ is indecomposable, whp.  (ii) Let $\mu_n\to -\infty$; then $C(\bs(n,m))\to \infty$, in
probability; that is $\bs(n,m)$ is highly decomposable, whp.  Thus $m_n:=
(6n/\pi^2)\log n$ is a sharp threshold for transition of $\bs(n,m)$ from being decomposable to
being indecomposable, whp.

The proof of Theorem \ref{main} is a culmination point of a series of intermediate claims.

Let $\bs\in S_n$ and let $\bs= (\bs^1,\bs^2,\dots, \bs^l)$ be its decomposition into indecomposable blocks.  The inversion sequence of $\bs$ is  the concatenation of the inversion sequences of $\bs^i$'s, i.e., we have  $\boldsymbol x (\bs) = \boldsymbol x (\bs^1),\boldsymbol x (\bs^2),\dots, \boldsymbol x(\bs^t)$. Using this decomposition of $\bs$  we define the function 
\[ \psi= \psi_n : S_n \longrightarrow S_n \]
such that the inversion sequence of the permutation $\psi(\bs)$ is given by $\bX(\bs^t)\bX(\bs^{t-1})\dots \bX(\bs^1)$, i.e., we have 
\be	\label{s3:e8}
\bX(\psi(\bs))= \bX(\bs^t)\bX(\bs^{t-1})\dots \bX(\sigma^1).	
\ee
Note that the function $\psi$ is a bijection from $S_n$ onto itself. Indeed, the inverse of $\psi$ is itself. It follows directly from the construction of the function that the number of inversions is invariant under $\psi$, so $\psi$ is a bijection from $\mathcal{S}(n,m)$ onto itself. In
particular, $\psi(\bs(n,m))$ is distributed uniformly on $\mathcal{S}(n,m)$.

We begin with  some symmetry results.
Let $A_i=\{\bs\,:\, \bs\text{ is decomposable at the position }i\}$.
\begin{lemma}\label{symmetry}
Let $r$ be a positive integer, and let $1\leq i_1< i_2< \cdots < i_r \leq n-1$. Then, 
\[	\bbP \bigg( \displaystyle \bigcap \limits _{j=1}^r A_{i_j}\bigg)	= \bbP\bigg( \displaystyle \bigcap \limits _{j=1}^r A_{n-i_j}\bigg)  \]
\end{lemma}

\begin{proof}
It follows from equation ~\eqref{s3:e8} that  $\bs \in \displaystyle \bigcap \limits _{j=1}^s A_{i_j}$, 
if and only if $\psi(\bs) \in \displaystyle \bigcap \limits _{j=1}^s A_{n-i_j}$.
\end{proof}

\begin{cor}\label{EventA}
For an integer $1\leq k \leq \frac{n}{2}$, we have 
\be	\label{s3:e9}
\bbP\bigg( \displaystyle \bigcup \limits _{j=1}^k A_{j}\bigg)	= \bbP\bigg( \displaystyle \bigcup \limits _{j=1}^k A_{n-j}\bigg).		
\ee
\end{cor}

\begin{proof}
The claim follows from Lemma \ref{symmetry} and the inclusion-exclusion formula
applied separately to both sides of  \eqref{s3:e9}.
\end{proof}

To proceed,  define $\al:=\frac{m}{n}$ and introduce $I:= \big[ 0.5 \log n,\log n\big]$.
Unless otherwise stated, we will assume that $\al\in I$.  Equivalently, the number of inversions $m$ lies in the interval $\big[ 0.5 n\log n, n\log n\big]$ unless otherwise stated.

Introduce $\bX=X_1\dots X_n=\boldsymbol x(\bs(n,m))$, the inversion sequence of $\bs(n,m)$. In
view of bijection between $\mathcal S(n,m)$ and $\mathcal X(n,m)$, the set of sequences 
$\boldsymbol x$ meeting the constraints ~\eqref{s1:e1}--\eqref{s1:e2} from Section ~\ref{Intro},
$\boldsymbol X$ is uniformly distributed on $\mathcal X(n,m)$. 

Next  we define $\nu=\lc  2 (\alpha+1) \log n \rc$ and introduce $\boldsymbol X^{\nu}=
X_{\nu+1}\dots X_n$, a tail of $\boldsymbol X$. We want to show that whp the number of (indecomposable) blocks in $\boldsymbol X^{\nu}$ is the same as that in $\boldsymbol X$ itself. Let $a_1,a_2,\ldots,a_{\nu}$ be integers such that $0\leq a_i \leq i-1$ for $i=1,2,\dots,\nu$, and let $a:=a_1+\cdots +  a_{\nu}$; in
particular, $a\leq \binom{\nu}{2}\ll m$.  Let $X_i=a_i$ for $1\leq i\leq \nu$. Then $Y_i:=X_{\nu+i}$
satisfy the constraints $0\le Y_i\le \nu+i-1.$
Introduce the following three sets of sequences $\boldsymbol y=y_1\dots y_{n-\nu}$:
\begin{align*}
\mathcal{Y}_1= \mathcal{Y}_1(\boldsymbol a) &:= \left\{ \boldsymbol y: \dsum_{i=1}^{n-\nu}y_i=m-a,\  0\leq y_i, \ \forall i \right\},	\\
\mathcal{Y}_2= \mathcal{Y}_2(\boldsymbol a) &:= \left\{ \boldsymbol y: \dsum_{i=1}^{n-\nu}y_i=m-a, \  0\leq y_i \leq \nu+i-1, \  \forall i \right\},	\\
\mathcal{Y}_3= \mathcal{Y}_3(\boldsymbol a) &:= \left\{ \boldsymbol y: \dsum_{i=1}^{n-\nu}y_i=m-a, \ 0\leq y_i \leq \nu-1, \  \forall i \right\}.
\end{align*}
From the definition, it is clear that $\mathcal{Y}_1 \supset \mathcal{Y}_2 \supset \mathcal{Y}_3$. Note that, if we take a sequence from $\mathcal{Y}_2$ and append it to $\boldsymbol a$ from the right,  then we get an inversion sequence $\boldsymbol x\in \mathcal X(n,m)$. Conversely, any inversion sequence in $\mathcal X(n,m)$ starting with $\boldsymbol a$ ends with a sequence $\boldsymbol y\in \mathcal{Y}_2$. 

\begin{lemma}\label{threesets}
Uniformly for $\boldsymbol a$ as defined above,
\[ 		\frac{|\mathcal{Y}_3|}{|\mathcal{Y}_1|}=1-O(n^{-1}).		\]
Thus almost all $\boldsymbol y\in\mathcal{Y}_1$ are in $\mathcal{Y}_3$.
\end{lemma}

{\bf Note.\/} $\mathcal{Y}_1$ is just  the set of all compositions of $m-a$ into $n-\nu$ 
non-negative integer parts, 
and as such it is considerably easier to deal with than $\mathcal{Y}_2$, the set of all
tails of the inversion sequences $\boldsymbol x$, with $\boldsymbol a$ being their  first $\nu$ components.
And, denoting the uniform probability measure on $\mathcal{Y}_i$ by $\bbP_{\mathcal{Y}_i}$,
for any set $Q$ of compositions of $m-a$, we have
\[
\frac{|Q\cap \mathcal{Y}_1|}{|\mathcal{Y}_1|}=\frac{|Q\cap \mathcal{Y}_2|}{|\mathcal{Y}_2|}
\cdot\frac{|\mathcal{Y}_2|}{|\mathcal{Y}_1|}+\frac{|Q\cap (\mathcal{Y}_1\setminus 
\mathcal{Y}_2)|}{|\mathcal{Y}_1|}.
\]
So, by Lemma \ref{threesets},
\begin{equation}\label{PY1=PY2}
\bigl|\bbP_{\mathcal{Y}_1}(\boldsymbol Y\in Q)-\bbP_{\mathcal{Y}_2}(\boldsymbol Y\in Q)\bigr|=O(n^{-1}),
\end{equation}
uniformly for all $Q$.
\begin{proof}[Proof of Lemma \ref{threesets}]
Let $\boldsymbol Y=Y_1\dots Y_{n-\nu}$ be chosen uniformly at random from $\mathcal{Y}_1$. Let 
\[
B=\#\{1\le i\le n-\nu: Y_i\ge \nu\}.
\]
Since $\bbP_{\mathcal{Y}_1}(B=0)=|\mathcal{Y}_3|/|\mathcal{Y}_1|$, 
we need to prove $\bbP_{\mathcal{Y}_1}(B=0)= 1-O(n^{-1})$ as $n\to\infty$.  It is enough to show that 
$E_{\mathcal{Y}_1}[B] =O(n^{-1})$ as $n \to \infty$. Since the number of compositions of an integer $\mu$ into $\ell$ non-negative integer 
parts is $\binom{\mu+\ell-1}{\ell-1}=(\mu+\ell-1)_{\ell-1}/(\ell-1)!$, we compute
{\allowdisplaybreaks
\begin{align*}
E_{\mathcal{Y}_1}[B] 	&= (n-\nu)\, \bbP_{\mathcal{Y}_1}( Y_1 \geq \nu) \leq n\bbP_{\mathcal{Y}_1}( Y_1 \geq \nu)\\
	& = n\, \dfrac{(m-a-\nu+n-\nu-1)_{n-\nu-1}}{(m-a+n-\nu-1)_{ n-\nu-1}} \\
	& \leq n \left( \frac{m-a+n-2\nu-1}{m-a+n-\nu-1}\right) ^{n-\nu-1} 
	 \leq n\left(1- \frac{\nu}{m-a+n-\nu-1}\right)^{n-\nu-1}\\ 
	&\leq n\,\exp\left(-\frac{\nu(n-\nu-1)}{m-a+n-\nu}\right)=n\exp\left(-\frac{\nu n}{m}\bigl(1+O(1/\log n)    \bigr)\right)\\
	&=  \exp\left(\log n -2(\log n)\frac{m}{n} \frac{n}{m}+O(1)\right)=O(n^{-1}). \qedhere
	\end{align*}
}
\end{proof}

In light of Lemma \ref{threesets} and the note following it, we focus on the asymptotic
properties of $\bbP_{\mathcal{Y}_1}$. In what follows, we use the notation 
\[	q:=\frac{\alpha}{\alpha+1}.		\]
For the magnitudes of the quantities we have defined so far see Figure \ref{figure3}.
\begin{figure}\caption*{} 
\begin{center}
\begin{tabular}{|l|ll|}
\hline
Quantity              & Order &	(as $n \to \infty$)			\\
\hline
$m$ 					& 	$\Theta(n\log n)$  		&		\\
$\alpha=m/n$           		& 	$\Theta(\log n)$  		&		\\
$\nu\sim 2\alpha\log n$   	&	$\Theta(\log^2 n)$ 	&		\\
$a$						&	$O(\log^4n)$			& 		\\
$q=\alpha/(\alpha+1)$    	& 	$\sim 1$ 			&					\\
\hline
\end{tabular}
\caption{The orders of magnitude of critical quantities appearing in the paper.}\label{figure3}
\end{center}
\end{figure}

\begin{lemma}\label{Geom}
 Let $t$ be a positive integer,  and $d_1, d_2,\dots,d_t$ be non-negative integers such that 
 $t\leq \sqrt{n}/(\log n)^2$, $0\leq d_j \leq \nu$ for $1 \leq j\leq t$. Then,
 for $1\leq i_1<i_2<\cdots<i_t\leq n-\nu$,
\[
\bbP_{\mathcal{Y}_1}\bigl(Y_{i_1}=d_1, Y_{i_2} =d_2, \dots, Y_{i_t}=d_t\bigr) = \lb 1+ O\lp \frac{\nu^3t+\nu^2t^2}{m}\rp \rb  \prod_{j=1}^t (1-q)q^{d_j}.
\]
\end{lemma}

{\bf Note.\/} Probabilistically, Lemma \ref{Geom} asserts that the joint
distributions of the components $Y_1,\dots,Y_{n-\nu}$, of dimension $t\leq \sqrt{n}/(\log n)^2$, are close to those of $(n-\nu)$ independent geometrics with success probability $1-q$, as long as the generic values of $Y_j$'s are of order $(\log n)^2$, at most.

\begin{proof}[Proof of Lemma \ref{Geom}]
By the definition of $\mathcal{Y}_1$ and $\bbP_{\mathcal{Y}_1}$,
\begin{flalign*}
&\bbP_{\mathcal{Y}_1}\bigl(Y_{i_1}=d_1, Y_{i_2}=d_2, \dots, Y_{i_t}=d_t\bigr) = \frac{{(m-a-d)+(n-\nu-t)-1 \choose n-\nu-t-1}}{{m-a+(n-\nu)-1 \choose n-\nu-1}} \\
=&\prod_{i=0}^{d-1}\frac{m-a-i}{m-a+n-\nu-1-i}\,\prod_{j=1}^{t}\frac{n-\nu-j}{m-a+n-\nu-d-j},
\end{flalign*}
where $d:=d_1+\cdots+d_t$. 
Using the inequalities
\[	i< d \leq \nu t, \quad  j\leq t, \quad a\leq \frac{\nu^2}{2},	\]
we get
\[	\frac{m-a-i}{m-a+n-\nu-1-i}=\frac{m}{m+n} \lb 1+O\lp \frac{\nu^2+\nu t}{m}\rp\rb, 	\]
and
\[	\frac{n-\nu-j}{m-a+n-\nu-d-j}=\frac{n}{m+n} \lb 1+O\lp \frac{\nu^2+\nu t}{m}\rp\rb,		\]
uniformly for $i$ and $j$ in question. Then,
{\allowdisplaybreaks
\begin{align} \label{s3:e6}
\prod_{i=0}^{d-1}\frac{m-a-i}{m-a+n-\nu-1-i} &= \lp \frac{m}{m+n}\rp^d \lp 1+O\lp \frac{\nu^2+\nu t}{m}\rp \rp^d \notag \\
&= \lp \frac{m}{m+n}\rp^d \lb 1+O\lp \frac{\nu^3t+\nu^2t^2}{m}\rp \rb ,
\end{align}
}	
and similarly,
\begin{align} \label{s3:e7}
\prod_{j=1}^{t}\frac{n-\nu-j}{m-a+n-\nu-d-j} &= \lp \frac{n}{m+n}\rp^t \lp 1+O\lp \frac{\nu^2+\nu t}{m}\rp \rp^t 	\notag	\\
&= \lp \frac{n}{m+n}\rp^t \lb 1+O\lp \frac{\nu^2t+\nu t^2}{m}\rp \rb. 
\end{align}
Combining Eqs. \eqref{s3:e6} and \eqref{s3:e7} we get the desired result.
\end{proof}

\begin{cor}\label{ProbY}
Let $t$, $0 \leq d_i\leq \nu$, ($i\le t$), be integers as in Lemma \ref{Geom}.  Then, 
\[	 \bbP_{\mathcal{Y}_1}(Y_{i_1}\leq d_1, Y_{i_2}\leq d_2, \dots, Y_{i_t}\leq d_t)= \lb 1+O\lp \frac{\nu^3t+\nu^2t^2}{m}\rp \rb \prod_{j=1}^t \bigl(1-q^{d_j+1}\bigr).	\]
\end{cor}
\begin{proof}
By Lemma \ref{Geom},
{\allowdisplaybreaks
\begin{align*}
&\bbP_{\mathcal{Y}_1}(Y_{i_1}\leq d_1, Y_{i_2}\leq d_2, \dots, Y_{i_t}\leq d_t) \\
=& \sum_{j_1=0}^{d_1}\cdots \sum_{j_t=0}^{d_t}	\bbP_{\mathcal{Y}_1}(Y_{i_1}= j_1, \dots, Y_{i_t}= j_t)	\\
=&  \lb 1+ 
O\lp \frac{\nu^3t+\nu^2t^2}{m}\rp \rb\sum_{j_1=0}^{d_1}(1-q)q^{j_1}\cdots \sum_{j_t=0}^{d_t} (1-q) q^{j_t}\\
=& \lb 1+O\lp \frac{\nu^3t+\nu^2t^2}{m}\rp \rb \prod_{i=1}^t \lp 1-q^{d_i+1}\rp. \qedhere
\end{align*}
} 
\end{proof}

We will use Lemma \ref{threesets} and Corollary \ref{ProbY}  to prove that $\bs(n,m)$, or equivalently its inversion sequence $\boldsymbol X$, is not decomposable at the positions $n-\nu,n-\nu+1,\dots,n-1$. We start with the following lemma.

\begin{lemma}\label{EventB}
With  probability $1-O\bigl((\log n)^{-1}\bigr)$, the random composition $\boldsymbol Y=Y_1Y_2\dots Y_{n-\nu}\in \mathcal{Y}_1$ is not decomposable at every one of the positions $n-2\nu,n-2\nu+1,\dots,n-\nu-1$.
\end{lemma} 

\begin{proof}
Let $B_i$ be the event that the sequence $Y_1Y_2\dots Y_{n-\nu}\in \mathcal{Y}_1$ is decomposable at $i$. We need to prove that $\bbP_{\mathcal{Y}_1} \lp \bigcup_{\ell=0}^{\nu-1}B_{n-2\nu+\ell}\rp \to 0$. It suffices to show that  
\[
\sum_{\ell=0}^{\nu-1}\bbP_{\mathcal{Y}_1} (B_{n-2\nu+\ell}) \to 0.
\]
 By the definition of decomposability,
\[	 B_{n-2\nu+\ell} =\bigcap_{j=n-2\nu+\ell+1}^{n-\nu}\{Y_j\le j-(n-2\nu+\ell+1)\},
\quad \ell=0,1,\dots,\nu-1.\]
By Corollary ~\ref{ProbY} we have
\[	 \bbP_{\mathcal{Y}_1} (B_{n-2\nu+\ell})\le 2 \prod_{j=1}^{\nu-\ell} \lp 1-q^j \rp,	\]
uniformly for all $\ell$  in the range. Then
\begin{align*}
&\sum_{\ell=0}^{\nu-1}\bbP_{\mathcal{Y}_1} (B_{n-2\nu+\ell})	\le 2		\sum_{\ell=0}^{\nu-1}\prod_{j=1}^{\nu-\ell} \lp 1-q^j \rp = 2\sum_{d=1}^{\nu}\prod_{j=1}^{d} \lp 1-q^j \rp\\
\le&\,2\bigl[(1-q)+(1-q)(1-q^2)+(\nu-3)(1-q)(1-q^2)(1-q^3)\bigr]\\
=&\,O(\alpha^{-1})+O(\alpha^{-2})+O(\nu\alpha^{-3})=O(\alpha^{-1})=O((\log n)^{-1}),
\end{align*}
which finishes the proof.
\end{proof}

\begin{cor}\label{lastk}
With probability $1-O((\log n)^{-1})$, the tail $X_{\nu+1}X_{\nu+2}\dots X_n\,(\in \mathcal{Y}_2)$ of $\bs(n,m)$ is not decomposable at any of the positions $n-\nu, n-\nu+1,\dots, n-1$.
\end{cor}

\begin{proof} This follows immediately from Lemma \ref{EventB} combined with Lemma ~\ref{threesets}.
\end{proof}

\begin{cor}\label{first-last-k}
With probability $1-O((\log n)^{-1})$, the random permutation $\bs(n,m)$ is not decomposable at any of the positions $1,2,\dots,\nu$ and $n-\nu,n-\nu+1,\dots,n-1$.
\end{cor}

\begin{proof}
It suffices to show that
\[	\bbP\bigg(\bigcup_{i=1}^{\nu}A_i\bigg)+\bbP \bigg( \bigcup_{i=1}^{\nu}A_{n-i}\bigg) =O\bigl((\log n)^{-1}\bigr),	\]
where $A_i$ is the event that $\bs(n,m)$ is decomposable at position $i$. By Corollary ~\ref{EventA},  we have 
\[	\bbP\bigg(\bigcup_{i=1}^{\nu}A_i\bigg)= \bbP\bigg( \bigcup_{i=1}^{\nu}A_{n-i}\bigg),	\]
and by Corollary \ref{lastk},
\[	\bbP\bigg(\bigcup_{i=1}^{\nu}A_{n-i}\bigg) =O\bigl((\log n)^{-1}\bigr),		\]
from which we get the desired result.
\end{proof}

So far we have shown that $\bs(n,m)$, equivalently its inversion sequence $\bX$, is not decomposable at the positions $j\in \{1,2,\dots,\nu\} \cup \{n-\nu,\dots,n-1\}$ whp as $n \to \infty$. Then, whp, the number of indecomposable blocks of $\bs(n,m)$ is equal to the number of blocks in each of the sequences $X_{\nu+1}X_{\nu+2}\dots X_{n-\nu}$ and $X_{\nu+1}X_{\nu+2}\dots X_n$ as $n$ tends to infinity. By Lemma  \ref{threesets}
(see Eq. ~\eqref{PY1=PY2}), the distribution of the number of indecomposable blocks in $X_{\nu+1}X_{\nu+2}\dots X_n$ is  within max-distance $O\bigl(n^{-1}\bigr)$ from the distribution  of the number of indecomposable blocks in $\boldsymbol Y=Y_1Y_2\dots Y_{n-\nu}\in \mathcal Y_1$. Thus, it is enough to consider the random sequence $\boldsymbol Y\in \mathcal{Y}_1$.
\smallskip

\begin{definition}\label{markedpointdef}
Let $y_1\dots y_{n-\nu}\in \mathcal{Y}_1$. A point $ i \in [n-2\nu]$ is called a \textit{marked point} if $y_{i+t}\leq t-1$ for all $t \in [\nu]$. 
\end{definition}

The probability that a point $i$ is marked in $\bY$ is asymptotically $\prod_{j=1}^{\nu}(1-q^j)$ by Corollary~\ref{ProbY}, where $q=\alpha/(\alpha+1)$. We will use the following technical lemma later.

\begin{lemma}\label{Approximate h}
Let $h(q)= \prod_{j=1}^{\infty}(1-q^j)$, where $q=\frac{\al}{\al+1}$. Then, as $n\to \infty$,
\be \label{h,asymp}
	\prod_{j=1}^{\nu}(1-q^j)	= h(q)\cdot \lp 1+O(\alpha/n^2)\rp.				
\ee
\end{lemma}
 
\begin{proof}
 We have 
\begin{align*}
1\ge \prod_{j>\nu}(1-q^j)\ge&\, 1-\sum_{j\ge\nu}q^j=1-\frac{q^{\nu}}{1-q}	\notag	\\
=&\,1-(\alpha+1)\left(1-\frac{1}{\alpha+1}\right)^{\lc 2(\alpha+1)\log n \rc}=1-O(\alpha/n^2).
\end{align*}
Consequently, 
\[	\prod_{j=1}^{\nu}(1-q^j)	= h(q) \cdot	\prod_{j>\nu}\frac{1}{1-q^j} = h(q) \cdot \lp 1+O(\alpha/n^2)\rp.  \qedhere	\]
\end{proof}

Let ${\mcal M}$ and $D$ denote the number of marked points and the number of decomposition points, respectively, in the random sequence $\boldsymbol Y \in \mathcal{Y}_1$. The next proposition allows us to just focus on the marked points.

\begin{prop}\label{decomposition vs marked}
The set of decomposition points and the set of marked points in the random sequence $\boldsymbol Y \in \mathcal{Y}_1$ are the same with probability $1-O((\log n)^{-1})$. In particular,
\begin{equation}\label{dTV,D,M}
d_{\textup{TV}}\bigl(D,{\mcal M})=O((\log n)^{-1}),
\end{equation}
where $d_{\textup{TV}}(X,Y)$ is the total variation distance between the distributions of two
integer-valued random variables $X$ and $Y$.
\end{prop} 

\begin{proof}
 By Lemma ~\ref{EventB}, with probability $1-O((\log n)^{-1})$,  none of $i\in [n-2\nu+1,n-1]$ is a decomposition point. Also, if $i\le n-2\nu$  is a decomposition point, then it is also a marked point (see Definition~\ref{decompositionpoint}). On the other hand, if  $i$ is marked, then the sequence $\boldsymbol Y$ is decomposable at $i$ provided that $Y_j\leq \nu$ for $j=1,2,\dots,n-\nu$, and  by  Lemma ~\ref{threesets}, the latter holds with
probability $1-O(n^{-1})$.
\end{proof}

Let $\xi_j$ be indicator of the event $\{j\text{ is marked}\}$. Then,
\[
{\mcal M}=\sum_{j=1}^{n-2\nu}\xi_j.
\]
We will analyze a limiting distribution of ${\mcal M}$ via the binomial moments $E_{\mathcal{Y}_1}\left[\binom{{\mcal M}}{\ell}\right]$. 

\begin{lemma}\label{factorialmoments}
Let $h(q)= \prod_{j=1}^{\infty}(1-q^j)$, where $q=\frac{\al}{\al+1}$. Let $\kappa_n:=nh(q) + \frac{1}{nh(q)}$ and suppose that
\begin{equation}\label{nh+1/nh}
\kappa_n=o(\log\log n).
\end{equation}
Then, $\alpha\sim (6/\pi^2)\log n$ and, introducing $\varepsilon_n:=\frac{2\kappa_n}{\log\log n}\to 0$, for any  constant $\Delta \in (0,1)$,
\begin{equation*}  
\bbP_{\mathcal{Y}_1}({\mcal M}=j)= e^{-nh(q)}\,\frac{(nh(q))^j}{j!}\,\bigl[1+O((\log n)^{-\Delta
+\varepsilon_n})\bigr],
\end{equation*}
uniformly for $j\le (\log n)^{(1-\Delta)/2}$.
Consequently, for any constant $\Delta\in (0,1)$,
\begin{equation}\label{dTV,M,Poi}
d_{\textup{TV}}\bigl[{\mcal M},\textup{Poisson}(nh(q))\bigr]=O\bigl((\log n)^{-\Delta}\bigr).
\end{equation}
\end{lemma}

\begin{proof} 
That $\alpha\sim (6/\pi^2)\log n$ follows from a sharp asymptotic formula that will be proven in the next lemma. Consider the binomial moments of ${\mcal M}$.
Let $\boldsymbol i$ denote a generic $\ell$-tuple $(i_1,i_2,\dots,i_{\ell})$ with
$1\leq i_1<i_2<\cdots<i_{\ell}\leq n-2\nu$. Then
\be	\label{s3:e10}
E_{\mathcal{Y}_1}\left[\binom{{\mcal M}}{\ell}\right]=\sum_{\boldsymbol i} E_{\mathcal{Y}_1}[\xi_{i_1}\xi_{i_2}\cdots \xi_{i_{\ell}}] 
=\sum_{\boldsymbol i} \bbP_{\mathcal{Y}_1}[\xi_{i_1}=\xi_{i_2}=\cdots =\xi_{i_{\ell}}=1] .
\ee
Let $a_{\ell}=\nu$ and $a_k=\min\{\nu,i_{k+1}-i_k\}$ for $k<\ell$. 
The event $\{\xi_{i_1}=\xi_{i_2}=\cdots = \xi_{i_{\ell}}=1\}$ holds if and only if, for all $k\in [\ell],$
\[
Y_{i_k+j}\le j-1 \quad (\forall j\in [a_k]).
\]
By Corollary ~\ref{ProbY}, where the quantity $t$ of the corollary is at most $\ell\cdot \nu$, we obtain
\begin{align} \label{s3:e12}
E_{\mathcal{Y}_1}[\xi_{i_1}\xi_{i_2}\cdots \xi_{i_{\ell}}] &= \bbP_{\mathcal{Y}_1}
\left(\bigcap_{k=1}^{\ell}\bigl\{Y_{i_k+j}\leq j-1,\, \forall j \in [a_k]\bigr\}\right)	\notag 	\\
&= \bigl[1+O(\nu^4\ell^2/m)\bigr] \prod_{k=1}^{\ell}\prod_{j=1}^{a_k}(1-q^{j})
 \end{align}
for $\ell \leq (1/2)\cdot n^{1/2}\cdot (\log n)^{-4}$.

Given $S\subseteq [\ell]$, we call a tuple $\boldsymbol i$ of type $S$, if  $\{1\leq k\leq\ell: a_k = \nu \} = S$. We will show that the dominant contribution to the sum on the right hand side of equation~\eqref{s3:e10} comes from  tuples $\boldsymbol i$ of the type $S=[\ell]$. 
For a  tuple of type $[\ell]$, we have $a_k \equiv \nu$, and hence the double product on the right hand side of equation~\eqref{s3:e12} is  $\lp \prod_{j=1}^{\nu}(1-q^j)\rp^{\ell}$. Hence, using Lemma~\ref{Approximate h}, we get
\[ 	 E_{\mathcal{Y}_1}[\xi_{i_1}\xi_{i_2}\cdots \xi_{i_{\ell}}]=\bigl[1+O(\nu^4\ell^2/m)\bigr]\cdot
h^{\ell}	\cdot \lp 1+O(\alpha/n^2)\rp,	\]
where $h=h(q)= \prod_{j=1}^{\infty}(1-q^j)$. As $\alpha=\Theta(\log n)$, $\nu=\Theta(\log^2 n)$, and $m=\Theta(n\log n)$, we rewrite the previous equation as
\[
E_{\mathcal{Y}_1}[\xi_{i_1}\xi_{i_2}\cdots \xi_{i_{\ell}}]=\bigl[1+O(\ell^2\,n^{-1}\log^7n)\bigr]\cdot
h^{\ell}.
\]
Further,  let $1\leq i_1<i_2\cdots<i_{\ell}\leq n-2\nu$ be a tuple of type $[\ell]$. Setting $j_k=i_k-(k-1)(\nu-1)$, we have
\be \label{s3:e13}
1\leq j_1<j_2<\cdots <j_{\ell}\leq n-(\ell+1)\nu+(\ell-1) .					
\ee
Conversely any tuple $(j_1,\dots,j_{\ell})$ satisfying ~\eqref{s3:e13} gives rise to a type $[\ell]$ tuple. Therefore, the number of type $[\ell]$ tuples is 		
\[	{n-(\ell+1)\nu +(\ell-1)\choose\ell} = \frac{n^{\ell}}{\ell!} \lp 1-O\bigg( \frac{\ell^2 \nu}{n}\bigg)\rp.	\]
The contribution of type $[\ell]$ tuples to the sum in equation ~\eqref{s3:e10} is therefore asymptotic to 
\begin{equation}\label{s=ell}
\bigl[1+O(\ell^2\,n^{-1}\log^7n)\bigr]\cdot\frac{\bigl(nh\bigr)^{\ell}}{\ell!}.
\end{equation}
Now let $S$ be a proper subset of $[\ell]$. Let $i_1<i_2<\dots<i_{\ell}$ be a type $S$ tuple. By equation ~\eqref{s3:e12} and that $O(\nu^4\ell^2/m)=O(\ell^2n^{-1}\log^7n)$, we have
\begin{align} \label{expected-S-type}
E_{\mathcal{Y}_1}[\xi_{i_1}\xi_{i_2}\cdots \xi_{i_{\ell}}] & 	
= \bigl[1+O(\ell^2\,n^{-1}\log^7n)\bigr]\prod_{k=1}^{\ell}\prod_{j=1}^{a_k}(1-q^{j})	\\
&\le 2 \prod_{k \in S}\prod_{j_k=1}^{\nu}(1-q^{j_k}) \cdot \prod_{k \notin S}\prod_{j_k=1}^{a_k}(1-q^{j_k})	\notag	 \\
&\le 3  h^{s} \cdot \prod_{k \notin S}\prod_{j_k=1}^{a_k}(1-q^{j_k}), 	\notag
\end{align}
where $a_k=i_{k+1}-i_k<\nu$ and $s:=|S|$. The elements whose locations in $\boldsymbol i$
form $S$ and the set of $a_k$'s  together uniquely determine such a tuple. There are at most 
\[		{n-2\nu \choose s}\le {n \choose s}		\]
ways to choose those elements. Then,
 \begin{align*}
\sum_{\boldsymbol i\textrm{ is of type } S} E_{\mathcal{Y}_1}[\xi_{i_1}\xi_{i_2}\cdots \xi_{i_{\ell}}]	&\leq {n \choose s} 3h^{s} 	\prod_{k \notin S}\sum_{a_k=1}^{\nu-1}\prod_{j_k=1}^{a_k}(1-q^{j_k})		\\
&= {n \choose s}3h^{s}  f(q)^{\ell-s},
\end{align*}
where 
\be	\label{eq: defn f_a} f(q)= \sum_{a=1}^{\nu-1} f_a(q), \quad f_a(q)= \prod_{j=1}^{a}(1-q^j). \ee
Note that $f_a(q)$ is decreasing with $a$, and $f_a(q)=O((1-q)^a)$, for a fixed $a$. Then, 
\be \label{s3:e14}	f(q)= \sum_{a=1}^{\nu-1} f_a(q) = O \lp (1-q)+ (1-q)^2+\nu(1-q)^3 \rp	.	\ee
It follows from  equation~\eqref{s3:e14}, $1-q=O((\log n)^{-1})$, and $\nu=O(\log^2 n)$  that, for an absolute constant $c>0$,
\[	 f(q)\le\frac{c}{\log n}.	\]
Therefore, for  a proper subset $S$ of $[\ell]$,
\[	\sum_{\boldsymbol i \textrm{ is of type } S} E_{\mathcal{Y}_1}[\xi_{i_1}\xi_{i_2}\cdots \xi_{i_{\ell}}] \le {n\choose s}
 3h^{s} \left(\frac{c}{\log n}\right)^{\ell-s} \le 3\frac{(nh)^{s}}{s!} \left(\frac{c}{\log n}\right)^{\ell-s} . 	\]
Furthermore, given $1\le s\le\ell-1$, there are ${\ell \choose s}$ ways to choose a subset $S$ of size $s$. Then,
\[
\sum_{S\subset [\ell],\atop |S|=s} \ 	 \sum_{\boldsymbol i\textrm{ is of}\atop \textrm{ type } S} E_{\mathcal{Y}_1}[\xi_{i_1}\xi_{i_2}\cdots \xi_{i_{\ell}}]
\le \beta(s,\ell):=3\, \binom{\ell}{s}\frac{(nh)^s}{s!}\left(\frac{c}{\log n}\right)^{\ell-s}.
\]
Here,
\[
\frac{\beta(s,\ell)}{\beta(s-1,\ell)}=\frac{nh\log n}{c}\cdot\frac{\ell-s+1}{s^2}\ge\frac{nh\log n}{c}\frac{1}{\ell^2}\ge 2
\]
if
\begin{equation}\label{def,ell*}
\ell\le \ell^*:=\left\lfloor\sqrt{\frac{nh\log n}{2c}}\right\rfloor.
\end{equation}
Consequently, for $1\le \ell\le\ell^*$,
\begin{equation}\label{s<ell}
\sum_{S\subset [\ell],\atop |S|<\ell} \ 	 \sum_{\boldsymbol i\textrm{ is of}\atop \textrm{ type } S} E_{\mathcal{Y}_1}[\xi_{i_1}\xi_{i_2}\cdots \xi_{i_l}]
\le 6\binom{\ell}{\ell-1}\frac{(nh)^{\ell-1}}{(\ell-1)!}\left(\frac{c}{\log n}\right)=\frac{6c\ell^2}{nh\log n}\,\frac{(nh)^{\ell}}{\ell!}.
\end{equation}
Let
\[
E_{\ell}:=E_{\mathcal{Y}_1}\left[\binom{{\mcal M}}{\ell}\right].
\]
By equations \eqref{s=ell} and \eqref{s<ell}, we have
\[	E_{\ell}= 	\lb 1+O\lp \frac{\ell^2 \log^7n}{n}\rp  + O\lp \frac{\ell^2}{nh\log n}\rp \rb \cdot \frac{(nh)^{\ell}}{\ell!},	\]
uniformly for $1\le \ell \le \ell^*$.
By the assumption of the lemma (see equation~\eqref{nh+1/nh}) we have
\[	\frac{(\ell^2 \log^7n)/n}{\ell^2/(nh\log n)} = \frac{(nh)\cdot \log^8n}{n}= o\lp\frac{(\log\log n) \log^8n}{n} \rp	,		\]
and hence 
\be	\label{binom,ell,approx}	
E_{\ell}= 	\lb 1+ O\lp \frac{\ell^2}{nh\log n}\rp \rb \cdot \frac{(nh)^{\ell}}{\ell!}.
\ee

By Bonferroni's inequalities (see Feller \cite{Feller}, page 110) we have
\[
	\sum_{\ell=j}^{\ell+2k+1}(-1)^{\ell-j}\binom{\ell}{j} E_{\ell}	\leq	\bbP_{\mathcal{Y}_1} ({\mcal M}=j)	\leq	\sum_{\ell=j}^{\ell+2k}(-1)^{\ell-j}\binom{\ell}{j} E_{\ell}
\]
for any non-negative integer $k$. Thus, for $j<\ell^*$, we have
\begin{equation}\label{def,bonf}
\bbP_{\mathcal{Y}_1} ({\mcal M}=j)=\sum_{\ell=j}^{\ell^*-1}(-1)^{\ell-j}\binom{\ell}{j} E_{\ell}+
\mathcal{R}_{j},
\end{equation}
where
\[
|\mathcal{R}_{j}|\le \binom{\ell^*}{j} E_{\ell^*}\le 2^{\ell^*}E_{\ell^*}.
\]
It follows easily from equation \eqref{binom,ell,approx}, the definition of $\ell^*$,  condition \eqref{nh+1/nh}, and the inequality $(\ell^*)! \geq (\ell^*/e)^{\ell^*}$, that
\begin{equation}\label{calRj<}
|\mathcal{R}_{j}|\ll e^{-\sqrt{\log n}}.
\end{equation}
Next we need to bound the total contribution of the remainder terms $O\bigl[\ell^2/(nh\log n)\bigr]$ in
\eqref{binom,ell,approx} to the sum in \eqref{def,bonf}. Using
\[
\ell^2\le 2\bigl[j^2+(\ell-j)^2\bigr]=2\bigl[j^2+(\ell-j) +(\ell-j)_2],
\]
we have
\begin{align*}
\sum_{\ell\ge j} \ell^2\binom{\ell}{j}\,\frac{(nh)^{\ell}}{\ell!}
\le&\, 2\sum_{\ell\ge j}\frac{j^2+(\ell-j)+(\ell-j)_2}{j!(\ell-j)!}\,(nh)^{\ell}\\
=&\, 2e^{-nh}\,\frac{(nh)^j}{j!}\cdot e^{2nh}\bigl[j^2+nh+(nh)^2\bigr].
\end{align*}
So the absolute value of the contribution in question is at most of the order
\begin{equation}\label{cont<}
e^{-nh}\,\frac{(nh)^j}{j!}\cdot\frac{e^{2nh}}{\log n}\bigl[j^2/(nh)+1+nh\bigr].
\end{equation}
Since $nh \leq \kappa_n$,  we have
\[	\frac{e^{2nh}}{\log n}\leq \frac{e^{2\kappa}}{\log n}	 =	(\log n)^{-1+\varepsilon_n}.		\]
For $j$ satisfying
\begin{equation}\label{j<}
j\le (\log n)^{(1 -\Delta)/2},\quad\Delta\in (0,1),
\end{equation}
the sum in the square brackets is of order $(\log n)^{1-\Delta}\log\log n$.
Therefore, for $j$ satisfying \eqref{j<}, the expression \eqref{cont<} is of order
\begin{equation}\label{cont,expl}
e^{-nh}\,\frac{(nh)^j}{j!}\,\times\,(\log n)^{-\Delta+\varepsilon_n}.
\end{equation}
Combining \eqref{def,bonf}, \eqref{calRj<} and \eqref{cont,expl}, we get
\[
\bbP_{\mathcal{Y}_1} ({\mcal M}=j)=\sum_{\ell=j}^{\ell^*-1}(-1)^{\ell-j}\binom{\ell}{j} \frac{(nh)^{\ell}}{\ell!}
+O\left[e^{-nh}\frac{(nh)^j}{j!}\,\times\,(\log n)^{-\Delta+\varepsilon_n}\right].
\]
Finally,
\begin{equation*}
\sum_{\ell=j}^{\ell^*-1}(-1)^{\ell-j}\binom{\ell}{j} \frac{(nh)^{\ell}}{\ell!}=\,\sum_{\ell=j}^{\infty}(-1)^{\ell-j}\binom{\ell}{j} \frac{(nh)^{\ell}}{\ell!} - \sum_{\ell=\ell^*}^{\infty}(-1)^{\ell-j}\binom{\ell}{j} \frac{(nh)^{\ell}}{\ell!}
\end{equation*}
The first sum above is just $e^{-nh}(nh)^j/j!$. The second sum is bounded above by its first term since the sum is an alternating sum whose terms in absolute value decrease. Thus 
\begin{align*}
\sum_{\ell=j}^{\ell^*-1}(-1)^{\ell-j}\binom{\ell}{j} \frac{(nh)^{\ell}}{\ell!}=& e^{-nh}\,\frac{(nh)^j}{j!}+O\left[\binom{\ell^*}{j}\frac{(nh)^{\ell^*}}{\ell^*!}\right]\\
=&e^{-nh}\,\frac{(nh)^j}{j!}+o\bigl(e^{-\sqrt{\log n}}\bigr),
\end{align*}
and
\be \label{eq: small j}
\bbP_{\mathcal{Y}_1} ({\mcal M}=j)=e^{-nh}\,\frac{(nh)^j}{j!}\,\bigl[1+O((\log n)^{-\Delta
+\varepsilon_n})\bigr],\quad
j\le (\log n)^{(1-\Delta)/2}.			
\ee
Let $N:= (\log n)^{(1-\Delta)/2}$ and $Z$ be a Poisson random variable with mean $nh$. By equation~\eqref{eq: small j}, we have 
\be \label{DTV part1}
\sum_{j\leq N} \big|\bbP_{\mathcal{Y}_1}(\mcal M=j)- \bbP(Z=j)\big| = O\lp (\log n)^{-\Delta+\varepsilon_n}\rp.
\ee
Further, using a Chernoff-type bound,
\[
	\bbP(Z> N) = \bbP(e^{xZ}> e^{xN}) \le \frac{E[e^{xZ}]}{e^{xN}}= \frac{e^{nh(e^x-1)}}{e^{xN}},			
\]
for any $x>0$. Optimizing for $x$, we find that
\be 		\label{Chernoff}
	\bbP(Z> N)	\le e^{-nh}\frac{(enh)^N}{N^N}\ll (\log n)^{-t}, \quad \forall t>0.	
\ee
On the other hand,
\begin{align} \label{calM>N}
\bbP_{\mathcal{Y}_1}(\mcal M> N) &= 1 -\bbP_{\mathcal{Y}_1}(\mcal M\le N)	=1- \bbP(Z\le N)+O((\log n)^{-\Delta+\varepsilon_n})	\notag \\
&=\bbP(Z> N)+O((\log n)^{-\Delta+\varepsilon_n})	= O((\log n)^{-\Delta+\varepsilon_n}).
\end{align}
Combining \eqref{Chernoff} and \eqref{calM>N}, we obtain
\begin{align} \label{DTV part2}
\sum_{j> N} \big|\bbP_{\mathcal{Y}_1}(\mcal M=j)- \bbP(Z=j)\big| &\le \sum_{j> N} \big|\bbP_{\mathcal{Y}_1}(\mcal M=j)\big| + \sum_{j> N} \big|\bbP(Z=j)\big|  \notag	\\
&= O\lp (\log n)^{-\Delta+\varepsilon_n}\rp.
\end{align}
Equations \eqref{DTV part1} and \eqref{DTV part2} taken together imply: for any $\Delta \in (0,1)$,
\[	d_{\textup{TV}}\big[ \mcal M, \textup{Poisson}(nh)\big]= O\lp (\log n)^{-\Delta}\rp.	\qedhere		\]
\end{proof} 

The next lemma identifies the values of $\alpha$ for which the condition \eqref{nh+1/nh}
of Lemma \ref{factorialmoments} holds.

\begin{lemma}\label{technical2} Let 
\begin{equation}\label{alphafine}	
\al = \frac{6}{\pi^2}\lb \log n+\frac{1}{2}\log\log n+\frac{1}{2}\log(12/\pi)-\frac{\pi^2}{12}+\mu_n\rb,\quad |\mu_n|=o(\log\log n). 
\end{equation}	
Then
\begin{equation}\label{nh(q)=}
 nh(q)= \exp 
\lb -\mu_n +O\lp\frac{\log\log n}{\log n} \rp \rb,	 \textrm{ as } n\to \infty.
\end{equation}	
\end{lemma}

\begin{proof}

By Freiman's formula (see Pittel ~\cite{Pittel}),
\[	 \prod_{j=1}^{\infty}(1-e^{-jz}) =\exp\left[-\frac{\pi^2}{6z}-\frac{1}{2}\log\frac{z}{2\pi}+O(z)\right], \]
as $z \downarrow 0$. Then,
\begin{align}\label{s3:e1}
  h(q)= \dprod_{j=1}^{\infty} (1-q^j) 	&= \prod_{j=1}^{\infty}\lp1-e^{-j\log(1/q)}\rp \notag	\\
						&= 	\exp \lb -\frac{\pi^2}{6\log (1/q)}-\frac{1}{2}\log\frac{\log(1/q)}{2\pi}+O(1-q)\rb	
\end{align}
as $q \to 1$. Letting $q= \frac{\al}{\al+1}$, and using the Taylor expansion of logarithm we get
\[
\log(1/q)=\log\left(1+\frac{1}{\al}\right)= \frac{1}{\al} \lp 1- \frac{1}{2\al}+ O(\al^{-2})\rp.	
\]
Consequently
\[
\frac{1}{\log(1/q)}= \al + 1/2+O(\al^{-1})	
\]
and so we obtain
\be \label{s3:e4}
h(q)=	\exp \lb -\frac{\pi^2}{6}\al -  \frac{\pi^2}{12} +  \frac{1}{2}\log \al+	 \frac{1}{2}\log 2\pi+ O(\al^{-1})	 \rb.	
\ee
The formula \eqref{nh(q)=}  follows from plugging \eqref{alphafine} into
 equation~\eqref{s3:e4}, and multiplying the resulting expression by $n$. 
\end{proof}

\begin{cor}\label{DclosetoPoi} For
\[ \al = \frac{6}{\pi^2}\lb \log n+\frac{1}{2}\log\log n+\frac{1}{2}\log(12/\pi)-\frac{\pi^2}{12}+\mu_n\rb,\quad |\mu_n|=o(\log\log\log n)\]
we have
\begin{equation}\label{dTV,D,Poi} 
d_{\textup{TV}}(D,\textup{Poisson}(nh(q)))=O\bigl((\log n)^{-\Delta}\bigr),\quad\forall\,\Delta\in (0,1).
\end{equation}
\end{cor}

\begin{proof}
For the given $\alpha$ and $\mu_n$, condition \eqref{nh+1/nh} of Lemma~\ref{factorialmoments} holds by Lemma~\ref{technical2}. The proof of the corollary follows immediately from Lemma \ref{factorialmoments} and equation~\eqref{dTV,D,M}.
\end{proof}

\begin{proof}[\textbf{Proof of Theorem \ref{main}}]
The proof follows directly from Lemma~\ref{threesets},  Corollary~\ref{DclosetoPoi}, and Lemma~\ref{technical2}.
\end{proof}

\section{Block sizes in a near-critical phase}\label{se: block sizes}

We now turn our attention to the sizes of the largest and the smallest indecomposable blocks
for $m$ close to the threshold value  for indecomposability of $\bs(n,m)$ whp in
Theorem \ref{main}. Of course, that $m=m_n$ is also the threshold for
connectedness of the attendant permutation graph $G_{\bs(n,m)}$, thus it is a permutation
counterpart of the connectedness threshold for Erd\H{o}s-R\'{e}nyi graph $G(n,m)$. However
the component sizes behavior in $G_{\bs(n,m)}$ and $G(n,m)$, for $m$ relatively close
to the respective threshold from below, differ significantly. In $G(n,m)$, whp there is a single giant
component and a number of isolated vertices, i.e., components of size $1$. 
In this section we will show that, for $m$ close to that in Theorem \ref{main},  whp the length of the shortest block in $\bs(n,m)$ 
(i.e., the size of the smallest component in $G_{\bs(n,m)}$) is fairly large, and there is no 
component dwarfing in size all other components. To be precise, we consider the range
\be	\label{rangealpha}
\alpha=\frac{6}{\pi^2}\left[ \log n+0.5 \log \log n+\log(12/\pi)-\pi^2/12 +\mu_n \right]	
\ee
where $\mu_n\to -\infty$ such that $\mu_n=o(\log \log \log n)$.

\begin{definition}
The size of an indecomposable block is the number of letters in it. For example, the permutation $\bs=24135867$ has three indecomposable blocks, which are $2413$, $5$, and $867$,  and the respective sizes of the blocks are $4$, $1$, and $3$. In the rest of the paper we use the notation $L_{\text{first}}$, $L_{\text{last}}$, $L_{\text{min}}$, and $L_{\max}$ for the sizes of the first block, the last block, the shortest block, and the longest block, respectively. In the example above, we have $L_{\text{first}}=L_{\max}=4$, $L_{\text{last}}=3$, and $L_{\text{min}}=1$.
\end{definition}

Recall that a decomposition point indicates where an indecomposable block ends. In other words, any indecomposable block lies between two decomposition points, that is, if $i<j$ are two decomposition points of $\bs$ and there is no other decomposition point between them, then there is an indecomposable block starting with $\sigma(i+1)$ and ending with $\sigma(j)$. The size of this indecomposable block is $j-i$. By Corollary~\ref{first-last-k}, the permutation $\bs(n,m)$ is not decomposable at the first $\nu$ positions whp. Thus, it is enough to study the decomposition points in the tail $\bX^{\nu}=X_{\nu+1}\dots X_n$, where $X_1\dots X_n$ denotes the inversion sequence of $\bs(n,m)$.  As in the previous section, equation~\eqref{PY1=PY2}
enables us to focus on the uniformly chosen random sequence $\boldsymbol Y \in \mathcal{Y}_1$. By Proposition \ref{decomposition vs marked},  whp the set of decomposition points in $\boldsymbol Y$ is the same as the set of marked points, and therefore considering the locations of marked points suffices. Provided that the first block is not the smallest or the largest block, by Lemma ~\ref{threesets} and Corollary~\ref{first-last-k}, whp, $L_{\text{min}}$  and $L_{\max}$ are the same as the sizes of the smallest and the largest blocks in $\boldsymbol Y$, respectively.

The total number of marked points is asymptotically 
Poisson($nh$) by Lemma~\ref{factorialmoments}, where $nh\to\infty$, since $\mu_n\to-\infty$. Our guiding intuition is that, since Poisson($nh$) is sharply concentrated around $nh$, its expected value, the sizes of the smallest block and the largest block, scaled by $n$, should be asymptotically close to the lengths of the
shortest subinterval and the longest subinterval, respectively, in a partition of the unit interval $[0,1]$ by $r:=\lfloor nh\rfloor$ points chosen
uniformly, and independently of each other. It is known that those two lengths are
asymptotic, respectively, to $r^{-2}Y$, with $\bbP(Y\le y)=1-e^{-y}$, and to $r^{-1}(\log r+~Z)$, with
$\bbP(Z\le z)=e^{-e^{-z}}$, see for instance Holst~\cite{Holst} and Pittel~\cite{Pittel1989}. Consequently, we expect the size of the smallest block to be asymptotic to $n\cdot  Y(nh)^{-2}$, and the size of the largest block to be asymptotic to $n\cdot (nh)^{-1} (\log (nh)+Z)$, where the distributions of $Y$ and $Z$ are as given above.
\smallskip

We call a set of consecutive integers an interval. The interval $\{a,a+1,\dots,b\}$ is denoted by $[a,b]$. The length of the interval $[a,b]$ is the number of integers in it, that is, the length of $[a,b]$ is $b-a+1$.

\begin{definition}
Let $A$ be a subset of $[n-2\nu]$. We say that $A$ is marked when all of its elements are marked.
\end{definition}

\begin{lemma}\label{markedset}
Let $A=\{a_1,\dots,a_k\}$ with $1\le a_1<a_2<\cdots<a_k\le n-2\nu$, where $k$ is a fixed positive integer. Let $d_k=\nu$, and $d_i:=a_{i+1}-a_i$ for $i=1,\dots,k-1$. Then, 
\begin{enumerate}[(i)]
\item	$\bbP_{\mathcal{Y}_1}(A \text{ is marked}) =\bigl(1+O(m^{-1}\log^8 n)\bigr) \prod\limits_{i=1}^{k}
\prod\limits_{j=1}^{\min(d_i,\nu)}(1-q^j)$, and consequently	
\item    if $\min_i d_i\ge \nu$, then		
\[
\bbP_{\mathcal{Y}_1}(A \text{ is marked}) =	\bigl(1+O(m^{-1}\log^8 n)\bigr)h^k.
\]
where $h=h(q)=\prod_{k=1}^{\infty}(1-q^k)$, and $q=\alpha/(\alpha+1)$.
\end{enumerate}
\end{lemma}

\begin{proof}
Part $(i)$ follows directly from Corollary~\ref{ProbY}, and part $(ii)$ follows from $(i)$ combined with equation~\eqref{h,asymp}.
\end{proof}

Here is a crucial corollary.

\begin{cor}\label{noclosemarkedpair}
Whp, there is no pair of marked points $i$, $j=i+a$ with $1\le a \le \nu$. 
\end{cor}

\begin{proof}
We show that the expected number of such pairs goes to 0. Let $P_a$ be the number of those  pairs $i$, $j$ for $j=i+a$. Given $i$, by Lemma \ref{markedset}, the probability that both $i$ and $i+a$ are marked is asymptotic to
$	\prod_{k=1}^{a}(1-q^k)\prod_{l=1}^{\nu}(1-q^l)	$
by Corollary~\ref{ProbY} and
\[	\prod_{k=1}^{a}(1-q^k)\prod_{l=1}^{\nu}(1-q^l)\le 2h(q)\prod_{k=1}^{a}(1-q^k),		\]
by Lemma~\ref{Approximate h}.
 Therefore
\begin{equation*}
\EY[P_a] \le 2nh	\prod_{k=1}^{a}(1-q^k),
\end{equation*}
for all $a\le\nu$ and large enough $n$. Here, by Lemma~\ref{technical2}, we have $nh\sim e^{-\mu_n}$. Summing over $a$, we bound the expected number of pairs in question:
\begin{equation}\label{sumEPd}
\sum_{a=1}^{\nu}\EY[P_a] \le 2nh\sum_{a=1}^{\nu}f_a(q),			
\end{equation}
where $f_a(q)= (1-q)(1-q^2)\cdots(1-q^a)$. Since $f_a(q)$ is decreasing with $a$ and $f_k(q)=O((1-q)^k)$ for every fixed $k$, we have
\[	\sum_{a=1}^{\nu}f_a(q)= 	O\lp (1-q)+(1-q)^2+\nu(1-q)^3\rp	\]	
as in equation~\eqref{s3:e14}. Using $(1-q)=O((\log n)^{-1})$ and $\nu = \Theta((\log n)^2)$, it follows that the sum on the RHS
of \eqref{sumEPd} is $O\left( (\log n)^{-1} \right)$, whence
\[	\sum_{a=1}^{\nu}\EY[P_a]= O\left( nh (\log n)^{-1} \right)= O \left( \frac{e^{-\mu_n}}{\log n}	\right) = o(1). \qedhere	\] 
\end{proof}

Consider first the size of the smallest block. Let $y$ be a positive constant and define $d=d(y):=\lfloor y/(nh^2)\rfloor$. We first deal with the first and last blocks.
\begin{lemma}\label{Lfirst,Llast>} The sizes of the first and last blocks of $\bs(n,m)$ are greater than $d$ whp, i.e.,
\begin{equation*}
\lim_{n\to\infty}\bbP\{L_{\textup{first}},\,L_{\textup{last}}>d\}=1.
\end{equation*}
\end{lemma}

\begin{proof} Since $L_{\text{first}}$ and $L_{\text{last}}$ are equidistributed (see Lemma~\ref{symmetry}), it 
suffices to show that 
\[         \lim_{n \to \infty}\bbP(L_{\text{last}}>d)=1.	\]
Equivalently, it is enough to show that the last block of $\bX^{\nu}=X_{\nu+1}\dots X_n$ has size greater than $d$ whp, where $\bX^{\nu}$ denotes the tail of the inversion sequence of $\bs(n,m)$. Then, by Lemma~\ref{threesets}, Lemma~\ref{EventB}, and Proposition~\ref{decomposition vs marked}, it is enough to show that whp there is no marked point of $\bY \in \mathcal{Y}_1$ in the interval $[n-d, n-2\nu]$. The last assertion is immediate as the expected number of marked points in that interval
is of order $hd$, where
\[
hd\le \frac{y}{nh}\to 0. \qedhere
\]
\end{proof}

It remains to study the number of internal blocks of $\bs(n,m)$. As in the previous lemma, it is enough to consider the internal blocks of the random sequence $\bY \in \mathcal{Y}_1$. For a sequence $\y=y_1\dots y_{n-\nu} \in \mathcal{Y}_1$, we color a pair $(i,j)$ red if $\nu \le j-i \le d-1$ and both $i$ and $j$ are marked in $\y$. If there is no red pair in the random sequence $\bY$, then whp there is no block of size in $[\nu,d-1]$ of $\boldsymbol Y$ and consequently of $\boldsymbol X^{\nu}$. Let $R$ be the number of red pairs. Then,  by equation~\eqref{PY1=PY2} and Corollary~\ref{noclosemarkedpair}, the probability of the event
$	\{R=0, \ L_{\textup{min}} <d\}			$
approaches 0, and so
\be \label{R,L_min}
\lim \bbP\{L_{\text{min}}\ge d\}=\lim_{n\to\infty}\bbP_{\mathcal{Y}_1}\{R=0\}.
\ee

\begin{thm}\label{RedPoisson} For each $j$,
\[
\lim_{n\to\infty}\bbP_{\mathcal{Y}_1}\{R=j\}=e^{-y}\,\frac{y^j}{j!},
\]
i.e.,  $R$ is in the limit Poisson($y$). Consequently
\[
\lim_{n\to\infty}\bbP_{\mathcal{Y}_1}\left\{L_{\textup{min}}\ge \frac{y}{nh^2}\right\}=e^{-y}.
\]
\end{thm}

\begin{proof} We need to show that, for every fixed $k\ge 1$,
\begin{equation}\label{EbinomRk}
\lim_{n\to\infty}\EY\left[\binom{R}{k}\right]=\frac{y^k}{k!}.
\end{equation}
Introducing $1_{(i,j)}$, the indicator of $\{i<j\text{ is red}\}$, we have $R=\sum_{i<j}1_{(i,j)}$.
So, denoting by $\boldsymbol\tau$ a generic $k$-tuple $\{(i_1,j_1)\prec\dots\prec (i_k,j_k)\}$
($\prec$ standing for lexicographical order on the plane), we have
\[
E^k:=\EY \left[\binom{R}{k}\right]=\sum_{\boldsymbol\tau} \EY[1_{(i_1,j_1)}\cdots 1_{(i_k,j_k)}].
\]
To evaluate $E^k$ asymptotically, 
we write $E^k=E_1+E_2$ where $E_1$ is the contribution of 
 $\boldsymbol{\tau}$'s  satisfying 
\begin{multline}\label{easytau}
1\le i_1< i_2-(d+\nu)<i_3-2(d+\nu)\\
<\cdots< i_k-(k-1)(d+\nu) \le n-(k+1)\nu-kd+1,
\end{multline}
and 
$E_2$ is the contribution of the remaining tuples $\boldsymbol{\tau}$. If a $\boldsymbol\tau$
meets \eqref{easytau} then the intervals $[i_r,j_r]$ are all disjoint  with in-between
gaps of size $\nu$ at least, and $j_k\le n-2\nu$. The number of
summands in $E_1$ is the number of ways to choose $i_1,\dots,i_k$, i.e.,
\[	\binom{n-(k+1)\nu-kd+1}{k}	\sim \binom{n}{k},		\]
times 
\[	(d-\nu)^k	\sim d^k,		\]
the number of ways to choose the accompanying $j_1,\dots,j_k$. And each of the summands
in $E_1$ is asymptotic to $(h^2)^k$ by Lemma \ref{markedset}.
Therefore
\[		E_1 \sim \frac{(ndh^2)^k }{k!} \sim \frac{y^k}{k!}.						\]

It remains to show that  $E_2\to 0$. For a generic $\boldsymbol\tau$ contributing to 
$E_2$, we introduce a set $T=T(\boldsymbol\tau)$ that consists of all distinct points in 
$\boldsymbol\tau$, i.e., 
\[
T=T(\boldsymbol\tau)=\{e\,:\,e=i_r \text{ or } j_s,\text{ for some }r\le k,\,s\le k\}.
\]
Then, 
\[	E_2= \sum_{\boldsymbol\tau} \bbP_{\mathcal{Y}_1}(T(\boldsymbol\tau) \text{ is marked}).			\]
Uniformly over sets $T\subset [n-2\nu]$ ($|T|\le 2k$),  the number of $\boldsymbol\tau$'s such that $T=T(\boldsymbol\tau)$ is bounded as $n\to\infty$. So it is enough to show that 
\[	\sum_{T}\, \bbP_{\mathcal{Y}_1}(T \text{ is marked}) \to 0, \quad n\to\infty,			\]
where the sum is taken over all {\it eligible\/} $T$'s with $|T|=t$, $t\le 2k$. By eligibility of $T$
we mean a set of conditions $T$ needs to satisfy in order to correspond to a $k$-tuple 
$\boldsymbol\tau$. To identify one such condition, we write 
\[
T= \{1\le e_1<e_2<\cdots <e_t \le n-2\nu  \},
\]			
and define 
\begin{align*}	
d_s: &=e_{s+1}-e_s \quad (1\le s \le t-1),\\
u_s: &=
\begin{cases}
\min(d_s,\nu), & \text{if } s<t ;\\
\nu, & \text{if } s=t.
\end{cases}			
\end{align*}
Clearly, the set $T$ is uniquely determined by $e_t$, the rightmost point in $T$, and $\boldsymbol d= (d_1,\dots,d_{t-1})$.
We partition $[t]$ into three sets $A_1,A_2,A_3$  as follows: $t \in A_1$, and for $s<t$,
\[		s \in \begin{cases} A_1, & \text{if } d_s>d+\nu;\\
A_2, & \text{if } \nu\le d_s\le d+\nu;\\
A_3, & \text{if } 1\le d_s<\nu.
 \end{cases}			\]
We denote $a_j=|A_j|$, and $T_j=\{e_s:s\in A_j\}$, $j=1,2,3$.

\noindent \textbf{Claim}: A necessary condition for
$T$ to be eligible is that, the numbers $a_i$ must satisfy 
\begin{equation}\label{a1<a2+a3}
a_1\le a_2+a_3.
\end{equation}
Moreover, if the equality occurs, then $a_1=a_2=k$, all the even numbers in $[2k]$ are in $A_1$, all the odd numbers in $[2k]$ are in $A_2$, and $e_{t-1}>n-  2\nu-d+1$.

\begin{proof} 
Let $\boldsymbol \tau$ be a tuple contributing to $E_2$, and consider $T=T(\boldsymbol \tau)$. If $e_s=i_r$ for some $r\in~[k]$,  then $j_r-i_r<d$, so that $e_{s+1}-e_s<d$, whence
$s\in A_2\cup A_3$. Thus if $s\in A_1$ then $e_s=j_r$ for some $r\in [k]$. Since $i_r<j_r$,  we must have $e_{s-1} \ge i_r$, whence  $e_s-e_{s-1}\le j_r-i_r <d$.  Therefore, necessarily $s-1\in A_2\cup A_3$. This shows that $a_1\le a_2+a_3$.

Now suppose $i_{r+1}-i_r \le d+\nu$ for some $r \in [k-1]$. If $i_r=i_{r+1}$, then $j_{r+1}>j_r$, and both differences $j_{r+1}-j_r$ and $j_r-i_r$ are less than $d$. Thus, all the elements of $T \cap [i_r,j_r]$, in particular the first and the last elements, lie in  $A_2\cup A_3$. Therefore a consecutive set of elements of $T$ (at least 2) lie in $A_2\cup A_3$, and as a result the inequality in \eqref{a1<a2+a3} is strict. If $i_r<i_{r+1}$, then any point in $T\cap [i_r,i_{r+1}]$  lies in $A_2\cup A_3$, and again the inequality is strict. 

 Thus, if $a_1=a_2+a_3$, then for any $1\le r \le t-1$, $j_r-i_r \ge \nu$, and $i_{r+1}-j_r\ge \nu$. Hence all the elements are distinct, $t=2k$, the odd numbers in $[t]$ belong to $A_2$, and even numbers in $[t]$ belong to $A_1$. Since the elements of $T$ form a tuple that contributes to $E_2$, there must be a violation of \eqref{easytau}, and that is the violation of the last inequality. Then, $i_{k}=e_{t-1}> n-2\nu-d+1$.
\end{proof}

To generate such a (minimally) eligible set $T$, first we choose $a_1$, $a_2$, and $a_3$ such that $a_1\le a_2+a_3$, and $a_1+a_2+a_3=t$. Next, we partition $[t]$ into subsets $A_1$, $A_2$, and $A_3$ with given cardinalities and with the rule that, if an element lies in $A_1$, then the previous element must lie in $A_2\cup A_3$. Finally, we  choose the last element $e_t$ and the vector $\boldsymbol d$ according to the restrictions imposed by $A_1$, $A_2$, and $A_3$.  Note that the
total number of  choices in these steps does not depend on $n$. Hence it is enough to show that
$\sum_{T}\Bbb P(T\text{ is marked})$ coming from the eligible $T$'s  with  given, admissible,  $A_1$, $A_2$, and $A_3$ goes to $0$ as $n\to\infty$.
 We have
\[	\prod_{j=1}^{\mu} (1-q^j)\sim \prod_{j=1}^{\infty}(1-q^j)=h=h(q),			\]
uniformly for  $\mu \ge \nu$. If $A_1$, $A_2$, and the $d_i$'s corresponding to set $A_3$ are known, then
\[	\bbP_{\mathcal{Y}_1}(T \text{ is marked}) \sim h^{a_1}h^{a_2} \prod_{i \in A_3}\prod_{j=1}^{d_i}(1-q^j).			\]
Suppose first that $a_1<a_2+a_3$. Given $A_1$, $A_2$, and $A_3$, there are at most $n^{a_1}(d+\nu)^{a_2}$ ways to choose the elements of $T$ corresponding to $A_1$ and $A_2$. Taking sum over all values of $d_i$'s corresponding to set $A_3$,
\[	\sum_{T} \bbP_{\mathcal{Y}_1}(T \text{ is marked}) \le 2 n^{a_1}d^{a_2}h^{a_1}h^{a_2} \left(\sum_{i=1}^{\nu-1}\prod_{j=1}^{i}(1-q^j)\right)^{a_3} = O \left( (nh)^{(a_1-a_2)}(\log n)^{-a_3}\right).				\]
Since $nh\sim e^{-\mu_n}=o(\log n)$, and $a_1< a_2+a_3$, we have 
\[	(nh)^{a_1-a_2}(\log n)^{-a_3}\to 0.				\]
Now suppose $a_1=a_2+a_3$. Then, by the claim, for admissibility of $A_1,A_2$, and $A_3$
with cardinalities $a_1,a_2$ and $a_3$, it is necessary that $a_1=a_2=k$ and $a_3=0$. In this case $e_{t}\in A_1$, $e_{t-1}=i_k \in A_2$, and $e_t>e_{t-1}>n-2\nu-d+1$. Thus, there are at most $d$ choices for $e_t$. Then, there are at most $n^{k-1}(d+\nu)^{k+1}$ ways to choose the elements of $T$ corresponding to $A_1$ and $A_2$. By Lemma~\ref{markedset}, the probability that such a $T$ is marked is asymptotical to $h^{2k}$. Then, for these $A_1,A_2$ and $A_3$, 
\begin{multline*}
\sum_{T}\bbP_{\mathcal{Y}_1}(T\text{ is marked})=	O\lp n^{k-1}	d^{k+1}h^{2k} \rp\\
   =
 O\lp n^{-1}d\, (ndh^2)^k\rp=O(n^{-1}d)= O\lp (nh)^{-2}\rp \to 0.	
 \end{multline*}			
In summary, we conclude that $E_2 \to 0$.

Since all the binomial moments of $R$ approach those of Poisson$(y)$, we conclude that $R$ approaches Poisson$(y)$, in distribution. Thus,
\[	\lim_{n\to\infty}\bbP_{\mathcal{Y}_1}\left\{L_{\text{min}}\ge \frac{y}{nh^2}\right\}=e^{-y}	\]
by equation~\eqref{R,L_min}.
\end{proof}

For the distribution of the size of the largest block, we define $d:= \lfloor (\log(nh)+z)/h \rfloor$, where $z$ is fixed real number. If a point $i \in [n-2\nu-d]$ is marked and the interval $[i+\nu+1,i+d]$ does not contain any marked point, then we color $i$ with blue. We denote by $B$ the number of blue points. Conditioned on the event that there is no pair of marked points within distance at most $\nu$, which happens whp by Corollary ~\ref{noclosemarkedpair}, $B$ counts the number of internal blocks (the blocks other than the first and last ones) of $\bY$ whose sizes exceed $d$. Thus, existence of a blue point implies existence of  a block in $\bY$ of size at least $d$ whp. Conversely, non-existence of a blue point implies  non-existence of an internal block whose size exceeds $d$. 

A key step  is to show that the number of blue points approaches a Poisson random variable. 
We begin with a lemma.

\begin{lemma}\label{technical3}
Let $\omega=\omega(n)$ be such that $\omega h \to \infty$. Let $k$ be a fixed positive integer, $\{i_1<\cdots< i_k\}\subset [n]$, and $I_j=[a_j,b_j], 1\le j \le k$, be intervals of length $\omega$ each, i.e., $b_j-a_j=\omega-1$. Let $i_1,\dots,i_k$
alternate with $I_1,\dots, I_k$ in such a way  that 
\begin{equation} \label{eq: conditions, i_j, I_j}
	a_{j}-i_j \ge \nu, \quad \forall j \in [k]; \qquad  i_{j+1}-b_j\ge \nu, \quad \forall j \in [k-1]; \quad  b_k\le n-2\nu.		
\end{equation}
 Let $M_j$ be the number of marked points in $I_j$ and $M= \sum_{j=1}^kM_j$. Then,
\[	\bbP_{\mathcal{Y}_1}\lp i_1,\dots,i_k \text{ are marked, } M=0 \rp  \sim h^ke^{-k\omega h}.	\]
\end{lemma}

\begin{proof}
Let $I$ be the union of the intervals $I_1,\dots,I_k$ and let
\[	E_{\ell}:=\EY \left[\binom{M}{\ell}\,\bigg|\,\{i_1,\dots,i_k\text{ are marked}\}\right].	\]
By definition
\be \label{p_j-i_j}
	E_{\ell}= \frac{1}{ \EY[ \xi_{i_1}\times \cdots \times \xi_{i_k}]}  \times	 \sum_{p_1<\cdots<p_{\ell} \atop p_j\in I\ \forall j}   \EY[\xi_{p_1}\times\cdots\times \xi_{p_{\ell}}\times \xi_{i_1}\times \cdots \times \xi_{i_k}],	
\ee
where $\xi_p$ is the indicator of the event that $p$ is marked. Each term of the sum in equation~\eqref{p_j-i_j} is the probability of  intersection of at most $(k+\ell)\nu$ events $\{Y_a\le b\}$ ($b<\nu$). In fact, by Corollary~\ref{ProbY} and conditions given in \eqref{eq: conditions, i_j, I_j}, we have: uniformly for $\ell \leq (1/2)\cdot n^{1/2}\cdot(\log n)^{-4}$,
\begin{align}\label{(k+ell) indicators}
\EY[\xi_{p_1}\times\cdots\times \xi_{p_{\ell}}\times \xi_{i_1}\times \cdots \times \xi_{i_k}]	&=  \bigl(1+O(m^{-1}\ell^2\log^8 n)\bigr)   \notag \\
  &\times       \lp    \prod_{r=1}^{\ell}\prod_{s=1}^{\eta_r}	(1-q^s) \rp	  \lp \prod_{j=1}^{\nu}(1-q^j)^k\rp ,
\end{align}
where $\eta_r:=\min(\nu,p_{r+1}-p_r)$ for $r<\ell$ and $\eta_{\ell}:=\nu$. Similarly, 
\begin{equation}\label{k indicators}
 \EY[ \xi_{i_1}\times \cdots \times \xi_{i_k}]= \bigl(1+O(m^{-1}\log^8 n)\bigr) \times  \prod_{j=1}^{\nu}(1-q^j)^k.
\end{equation}
Using Eqs. \ref{(k+ell) indicators} and \ref{k indicators} in equation \eqref{p_j-i_j}, we get
\[	E_{\ell}= \bigl(1+O(m^{-1}\ell^2\log^8 n)\bigr) \times   \sum_{p_1<\cdots<p_{\ell} \atop p_j\in I\ \forall j}   \   \prod_{r=1}^{\ell}\prod_{s=1}^{\eta_r}	(1-q^s).	\]
The rest is very similar to the proof of Lemma~\ref{factorialmoments} and we omit the details but note the crucial steps. We separate the set of points $\{p_1,\dots,p_{\ell}\}$ into two classes such that the first class consists of the points with all $\eta_a=\nu$ and the second class consists of the rest, and show that the main contribution comes from the first class of points. Note that, the total length of the union of the intervals $I$ is $k\omega$. For $\ell\le \varepsilon (k \omega h\log n)^{1/2}$ and $\varepsilon>0$ sufficiently small, similar to equation~\eqref{binom,ell,approx}, we obtain
\begin{align*}
 E_{\ell} = &\left[1+O\left(\frac{\ell^2}{\omega h\log n}\right)+O(m^{-1}\ell^2\log^8 n)\right]
\times \frac{ (k\omega h)^{\ell}}{\ell!}\\
=&\left[1+O\left(\frac{\ell^2}{\omega h\log n}\right)\right]\times\frac{(k\omega h)^{\ell}}{\ell!}.
\end{align*}
Arguing as in the end of the proof of Lemma~\ref{factorialmoments}, we conclude that for $M$, 
conditioned on the event $\{i_1,\dots,i_k\text{ are marked}\}$,
\[
d_{\textup{TV}}\bigl[M, \text{Poisson }(k\omega h)\bigr]=O\bigl((\log \omega)^{-\Delta}\bigr),\quad\forall\,\Delta\in (0,1).
\]
Therefore
\[
\bbP_{\mathcal{Y}_1}\big(M=0\,|\,\{i_1,\dots,i_k\text{ are marked}\}\big)=e^{-k\omega h} +O\bigl((\log \omega)^{-\Delta}\bigr),
\]
so that 
\[
\bbP_{\mathcal{Y}_1}(M=0\text{ and }i_1,\dots,i_k\text{ are marked})\sim h^k e^{-k\omega h}. \qedhere
\]
\end{proof}

\begin{lemma}\label{BluePoisson}
$B$ approaches  in distribution a Poisson random variable with mean $e^{-z}$ as $n \to \infty$.
\end{lemma}

\begin{proof}
As before, we need to evaluate  the binomial moments  of $B$. By definition,
\be \label{blue-fact-mom}	
\EY \left[ {B \choose k }\right] =	\sum_{\boldsymbol i} \bbP_{\mathcal{Y}_1}(\boldsymbol i \text{ is blue}),	
\ee
where the sum is over all $k$-tuples $\boldsymbol i =(i_1,\dots,i_k)$ such that 
\be \label{condition-k-tuples}	1\le i_1 <i_2-d<\cdots<i_k-(k-1)d	\le n-2\nu-kd.		\ee
We write
\[	\EY\left[ {B \choose k }\right] =	\sum_{\boldsymbol i} \bbP_{\mathcal{Y}_1}(\boldsymbol i \text{ is blue})= E_1+E_2				\]
where $E_1$ is the contribution of tuples $\boldsymbol i$ such that $i_j-i_{j-1}>d+\nu$, for $j=2,\dots,k$, and $E_2$ is the contribution of the remaining tuples $\boldsymbol i$. We will determine
$\lim E_1$, and show that $\lim E_2=0$. 

For a tuple $\boldsymbol i$ contributing to $E_1$, let $I_j:=[i_j+\nu+1,i_j+d]$ ($j=1,\dots,k$); $I_j$
has size $d-\nu$. Let $M_j$ be the number of marked points in $I_j$, and $M=\sum_j M_j$. Note that $\{i_1,\dots,i_k, I_1,\dots,I_k\}$ satisfy the conditions of Lemma~\ref{technical3}. Therefore
\[	\bbP_{\mathcal{Y}_1}(\boldsymbol i \text{ is blue})= \bbP_{\mathcal{Y}_1}( i_1,\dots,i_k \text{ are marked, } M=0) \sim	h^ke^{-k(d-\nu)h}	.		\]
Now set $x_j:=i_j-(j-1)(d+\nu)$. Then the numbers $x_1,\dots,x_k$ satisfy 
\[	1\le x_1<x_2<\cdots<x_k\le n-(k+1)\nu	-kd,		\]
so the number of tuples that contribute to $E_1$ is 
\[	{n-(k+1)\nu	-kd \choose k} 	\sim {n \choose k} \sim \frac{n^k}{k!}	.		\]
Thus
\[	E_1\sim 	\frac{n^k}{k!}	h^ke^{-k(d-\nu)h} \sim \frac{(nh)^k}{k!}e^{-k(\log (nh)+z)}	= \frac{e^{-kz}}{k!}.	\]

It remains to show that $E_2 \to 0$ as $n \to \infty$.
For a generic  $\boldsymbol i$ contributing to $E_2$, we now define the intervals 
$I'_j:=[i_j+\nu+1,i_j+d-\nu]$. The event that $\{\bi \text{ is blue}\}$ is contained in the event $\{\bi \text{ is marked, there is no marked point in } \cup_jI'_j\}$.  The length of each interval is $d-2\nu$, and the set 
$	\{i_1,\dots,i_k,I'_1,\dots,I'_k\}			$
satisfies the conditions of Lemma~\ref{technical3} , whence 
\[	\bbP_{\mathcal{Y}_1}(\bi \text{ is blue})\le 	2 h^ke^{-k(d-2\nu)h}\le 3 h^ke^{-kd h}.				\]
To bound $E_2$ we group the tuples $\boldsymbol i$ by their type 
\[	S=S(\boldsymbol i):=  \{j<k: i_{j+1}-i_j >d+\nu \},		\]
and note $S(\boldsymbol i)\subset [k-1]$ for $\boldsymbol i$ in question.  Note that there are at most $n^{|S|+1}(d+\nu)^{k-1-|S|}$ tuples of a given type $S$.  Thus, the number of tuples that contribute to $E_2$ is $O(n^{k-1}d)$. So
\[	E_2= O\left(	n^{k-1}dh^ke^{-kdh}	\right)= O\left(	n^{k-1}dh^k(nh)^{-k}e^{-kz}	\right)	=O(d/n).	\]
This finishes the proof.
\end{proof}

The rest is short.
\begin{lemma}\label{lastd_2}
Let $I\subset [n-2\nu]$ be an interval of length $(d-\nu-1)$ and let $N$ be the number of marked points in $I$. Then $N\geq 1$ whp.
\end{lemma}

\begin{proof}
We have
\[	\EY[N] =\sum_{a \in I}\bbP_{\mathcal{Y}_1}(a \text{ is marked}) \sim (d-\nu-1)h	\sim dh		\]
by Lemma~\ref{markedset}. On the other hand, by  calculations similar to those in the
proof of Lemma ~\ref{factorialmoments}, we find
\[	\EY\left[ {N \choose 2} \right] = {d-\nu-1 \choose 2}h^2+O\lp dh(\log n)^{-1}\rp.				\]
Consequently, $E[N^2] \sim h^2d^2$. Since, $dh \to \infty$ as $n \to \infty$, we conclude by Chebyshev's inequality that $N$ is positive whp.
\end{proof}

\begin{cor}
For a fixed real number $z$, the largest block of $\bs(n,m)$ has the following limiting distribution
\[ \lim_{n\to\infty}	\bbP\left(L_{\textup{max}}\le \frac{\log (nh)+z}{h}\right) = e^{-e^{-z}}.			\]
\end{cor}

\begin{proof}
We make use of the strong relation between $\bs(n,m)$ (or its inversion sequence $\bX$) and $\bY \in \mcal Y_1$. Recall that the set of decomposition points of $\bs(n,m)$ is the same as the set of decomposition points of $\bX$.
In other words, the blocks of $\bs(n,m)$ are the same as the blocks of $\bX$. On the other hand, the decomposition points of $\bX$ that are not contained in the set $[\nu]$ are the same as the decomposition points of $\bX^{\nu}=X_{\nu+1}\dots X_n$, i.e., almost all the blocks of $\bX$ and $\bX^{\nu}$ are the same.
By Lemma~\ref{threesets}, whp, $\bY=Y_1\dots Y_{n-\nu}$ and $\bX^{\nu}$ have the same block sizes.
By Lemma~\ref{decomposition vs marked}, whp, the decomposition points of $\bY$ are exactly the marked points of $\bY$. 
By Lemma~\ref{lastd_2}, whp, the random sequence $\bY$ has marked points in both of the intervals $[1,d-\nu-1]$ and $[n-d-\nu+1,n-2\nu-1]$.
Consequently, $\bX^{\nu}$ has decomposition points both in $[\nu+1, d-1]$ and $[n-d+1,n-\nu+1]$ whp. 
On the other hand, by Corollary~\ref{first-last-k}, $\bX$ does not contain any decomposition point in $[\nu]\cup [n-\nu,n-1]$ whp. 
As a result, both $L_{\textup{first}}$ and $L_{\textup{last}}$ lie in the set $[\nu+1,d-1]$. 
Then, whp, the internal blocks (blocks other than the first and last blocks) of $\bX$ are the same as the internal blocks of $\bX^{\nu}$, which have the same sizes as the internal blocks of $\bY$ whp. Then,
\[
\lim_{n\to\infty} \bbP(L_{\textup{max}}\leq d) = \lim_{n\to\infty} \bbP_{\mathcal{Y}_1}(B=0)= e^{-e^{-z}}, \quad d:=\left \lf \frac{\log (nh)+z}{h} \right \rf,
\]
where the second equality is due to Lemma~\ref{BluePoisson}. 
\end{proof}

\section*{Acknowledgment}
We are grateful to the editor for pinpointing expertly a badly phrased paragraph in the
initial version. The hardworking referees read the paper with painstaking care and provided us with many valuable comments and suggestions.

\bibliographystyle{amsplain}

\end{document}